\documentclass[a4paper]{amsart}

\usepackage{amsmath}
\usepackage{mathtools}
\usepackage{amssymb}
\usepackage{amsthm}
\usepackage{etoolbox}
\usepackage[colorlinks=true,linkcolor=blue,citecolor=blue,urlcolor=blue,citebordercolor={0 0 1},urlbordercolor={0 0 1},linkbordercolor={0 0 1}]{hyperref} 

\usepackage{float}
\usepackage{subcaption}
\usepackage{rotating}
\usepackage{multirow}
\usepackage{tikz}
\usepackage{tikz-cd}
\usepackage[all]{xy}
\usepackage[compress,capitalize]{cleveref}

\usepackage{comment}
\usepackage{enumitem}

\usetikzlibrary{positioning}

\newcommand{\relscale}{1}

\newenvironment{relbare}[1][0.32 * \relscale]
{\begin{tikzpicture}[baseline={([yshift=-.5ex]current bounding box.center)}]
\pgftransformcm{0}{-#1}{#1}{0}{\pgfpointorigin};
}
{
\end{tikzpicture}
}

\newenvironment{relmatrix}[3][0.32 * \relscale]
    {\left[ \begin{relbare}[#1] 
\useasboundingbox (1-0.5,1-.4) -- (#2 +.5, #3 +.4);
}
{
\end{relbare} \right]
}

\newcommand{\relgensize}{0.33}

\newcommand{\relcoeff}[3][]{\ifstrempty{#1}{\draw [fill] (#2, #3) circle (0.19)}{\node at (#2, #3) {\tiny $\mathbf{#1}$}};}
\newcommand{\relgen}[4]{\draw (#1 - \relgensize, #2 - \relgensize) rectangle (#3 + \relgensize, #4 + \relgensize);}
\newcommand{\relgenperm}[4]{\draw (#1 - \relgensize, #2 - \relgensize) rectangle (#3 + \relgensize, #4 + \relgensize); \draw[densely dotted] (#1 - \relgensize, #2 - \relgensize) -- (#3 + \relgensize, #4 + \relgensize);}
\newcommand{\relgenpermT}[4]{\draw (#1 - \relgensize, #2 - \relgensize) rectangle (#3 + \relgensize, #4 + \relgensize); \draw[densely dotted] (#1 - \relgensize, #4 + \relgensize) -- (#3 + \relgensize, #2 - \relgensize);}

\renewcommand{\det}{\operatorname{det}}
\newcommand{\perm}{\operatorname{perm}}

\DeclareMathOperator{\hook}{hook}

\DeclareMathOperator{\VP}{VP}
\DeclareMathOperator{\VNP}{VNP}
\DeclareMathOperator{\sgn}{sgn}
\DeclareMathOperator{\Sym}{Sym}
\DeclareMathOperator{\Frac}{Frac}

\DeclareMathOperator{\GL}{GL}

\DeclareMathOperator{\chartext}{char}
\newcommand{\co}{\colon}
\newcommand{\bk}{\Bbbk}
\newcommand{\tensor}{\otimes}

\newcommand\ZZ{\mathbb{Z}}

\newcommand\PP{\mathbb{P}}
\newcommand\bx{\mathbf{x}}
\newcommand\bS{\mathbf{S}}

\newcommand{\epf}{\qed \vspace{+10pt}}

        \theoremstyle{plain}
        \newtheorem{theorem}{Theorem}[section]
        \newtheorem{corollary}[theorem]{Corollary}
        \newtheorem{lemma}[theorem]{Lemma}
        \newtheorem{question}[theorem]{Question}
        \newtheorem{proposition}[theorem]{Proposition}
        \newtheorem{conjecture}[theorem]{Conjecture}

        \theoremstyle{definition}
        \newtheorem{definition}[theorem]{Definition}

        \newtheorem*{example*}{Example}

        \theoremstyle{remark}
        \newtheorem{remark}[theorem]{Remark}
        \newtheorem*{remark*}{Remark}  
        
\numberwithin{equation}{section}

\crefformat{section}{#2\S #1#3}

\title[Syzygies of the determinant and permanent]{Syzygies of the apolar ideals of the \\determinant and permanent}
\author[J.~Alper]{Jarod Alper}
\address{Department of Mathematics \\ University of Washington \\ Box 354350 \\ Seattle, WA 98195-4350, USA}
\email{jarod@uw.edu}
\author[R.~Rowlands]{Rowan Rowlands}
\address{Department of Mathematics \\ University of Washington \\ Box 354350 \\ Seattle, WA 98195-4350, USA}
\email{rowanr@uw.edu}
\subjclass[2010]{Primary 13D02; Secondary 13P20, 68Q17, 14L30}

\begin{document}

\begin{abstract}  We investigate the space of syzygies of the apolar ideals $\det_n^\perp$ and $\perm_n^\perp$ of the determinant $\det_n$ and permanent $\perm_n$ polynomials.  Shafiei had proved that these ideals are generated by quadrics and provided a minimal generating set.  Extending on her work, in characteristic distinct from two, we prove that the space of relations of $\det_n^{\perp}$ is generated by linear relations and we describe a minimal generating set. 
The linear relations of $\perm_n^{\perp}$ do not generate all relations, but we provide a minimal generating set of linear and quadratic relations.  For both $\det_n^\perp$ and $\perm_n^\perp$, we give formulas for the Betti numbers $\beta_{1,j}$, $\beta_{2,j}$ and $\beta_{3,4}$ for all $j$ as well as conjectural descriptions of other Betti numbers.  Finally, we provide representation-theoretic descriptions of certain spaces of linear syzygies. 
\end{abstract}

\maketitle
\section{Introduction}

This paper began as an investigation into the difference in complexity between the determinant and permanent polynomials by exploring homological properties of their apolar ideals.

To set up our notation, let $\bk$ be a field and $n$ be a positive integer.  Let $\mathbf{x} =(x_{i,j})_{1\leq i,j\leq n}$ denote an $n \times n$ matrix of indeterminates.  The determinant and permanent polynomials are defined as
\begin{align*}
\det_n & = \det(\mathbf x) =  \sum_{\sigma \in \bS_n} \sgn(\sigma) \, x_{1, \sigma(1)} x_{2, \sigma(2)} \dotsm x_{n, \sigma(n)}, \\
\shortintertext{and}
\perm_n & = \perm(\mathbf x) = \sum_{\sigma \in \bS_n} x_{1, \sigma(1)} x_{2, \sigma(2)} \dotsm x_{n, \sigma(n)}.
\end{align*}
respectively, where $\bS_n$ is the symmetric group on $n$ letters.  If we denote the $\bk$-vector space of $n \times n$ matrices as $M_n(\bk)$, then the polynomials $\det_n$ and $\perm_n$ are elements in the vector space $\Sym^n M_n(\bk)^{\vee}$ of homogenous polynomials on $M_n(\bk)$ of degree $n$.

Understanding the difference between the determinant and permanent polynomials is of central interest in theoretical computer science and, in particular, algebraic complexity theory.  See \cref{sec:motivation} for more details on the connection to complexity theory.

\subsection{Apolarity} \label{sec:apolarity}

In this paper, we will investigate homological properties of the apolar ideals of $\det_n$ and $\perm_n$.  We begin by recalling the definition of the apolar ideal. Let $W$ be a vector space over $\bk$ of dimension $n$ and $f(x_1, \dotsc, x_n) \in \Sym^d W^{\vee}$ be a homogenous polynomial on $W$ of degree $d$, where we have chosen a basis $x_1, \dotsc, x_n$ of $W^{\vee}$.  Let $R = \Sym^* W^{\vee}$ and $S = \Sym^* W$; we will identify $R$ with the polynomial ring  $\bk[x_1, \ldots, x_n]$, and $S$ with the  subring $\bk[x_1^{-1}, \ldots, x_n^{-1}]$ of the fraction field $\Frac(R)$ of $R$.  Using this identification, multiplication induces an $S$-module structure $\star \co S \times R \to R$ as follows: for $f \in R$ and $g \in S \subset \Frac(R)$, then $g \star f = gf$ if $gf \in R$ and $0$ otherwise. Note that $g$ acts like ``differentiation without coefficients'': that is, if $h \in R$ is a polynomial independent of $x_i$, then
\begin{gather*}
x_i^{-1} \star h x_i^n = h x_i^{n-1} = \frac{1}{n} \frac{\partial}{\partial x_i} h x_i^n \\
\shortintertext{for $n > 0$ as long as $n$ is invertible in $\bk$, and}
x_i^{-1} \star h = 0 = \frac{\partial}{\partial x_i} h.
\end{gather*}

The \emph{apolar ideal} of $f$ is the ideal
\begin{align*}
f^{\perp} & \coloneqq \left\{ g \, \mid \, g \star f = 0 \right\}  \subseteq S.
\end{align*}

The quotient $\Sym^* W / f^{\perp}$ is a graded Artinian Gorenstein algebra with socle in degree $d$.  A theorem of Macaulay \cite{macaulay} states that the assignment $f \mapsto f^{\perp}$ gives a one-to-one bijection between homogeneous polynomials $f \in \Sym^d W^{\vee}$ up to scaling and homogenous ideals $I \subset \Sym^* W$ such that the quotient $\Sym^d W / I$ is a graded Artinian Gorenstein algebra with socle in degree $d$.   In particular, the homogeneous ideal $f^{\perp}$ determines $f$ uniquely up to scaling.  

Thus one can attempt to distinguish the determinant and permanent polynomials via studying their apolar ideals.  Specifically, we ask:

\begin{question} What are the minimal graded free resolutions of $S/\det_n^{\perp}$ and \linebreak $S/\perm_n^{\perp}$?
\end{question}

This question was the starting point for our investigations.  It is a theorem of Shafiei \cite{shafiei} that for every $n \ge 2$, both the apolar ideals $\det_n^{\perp}$ and $\perm_n^{\perp}$ are minimally generated by ${n + 1\choose 2}^2$ quadrics.  These quadrics can be explicitly described --- see \cref{sec:shafiei} for a summary of Shafiei's work.  The main result of this paper determines the relations between these generators (i.e., the first syzygies):

\begin{theorem} \label{main-theorem} \, \hspace{1cm}
\begin{enumerate}[label=\alph*),ref=(\alph*)]
\item \label{main-theorem-a}
If $\chartext(\bk) \neq 2$, all relations of $\det_n^{\perp}$ are minimally generated by $4 {n+1 \choose 3} {n+2 \choose 3}$ linear relations.  
\item \label{main-theorem-b}
In arbitrary characteristic, all relations of $\perm_n^{\perp}$ are minimally generated by $4 {n+1 \choose 3} {n+2 \choose 3}$ linear relations and $2 {n \choose 2}{n \choose 4}$ quadratic relations.
\end{enumerate}
\end{theorem}

\begin{remark}
Observe that in characteristic 2, the determinant and permanent polynomials are equal, and thus the relations of $\det_n^{\perp}$ are described by Part \ref{main-theorem-b}. 
We find it interesting that in $\chartext(\bk)=2$, the syzygies of the apolar ideal of $\det_n$ seem to resemble those of the $\perm_n$ in arbitrary characteristic.
\end{remark}
As in Shafiei's result, we provide an explicit minimal generating sets of relations for both $\det_n^{\perp}$ and $\perm_n^{\perp}$; see \cref{thm:linear-relations-list,thm:linear+Q-generate-all-permanent}.

Moreover, we compute the dimension of the space of linear second syzygies of both $\det_n$ and $\perm_n$ (\cref{thm:beta34}) and conjecture on the dimension of linear higher syzygies (\cref{conj:betarr1}).  In \cref{sec:higher-syzygies}, we provide Macaulay2 computations of the full or partial Betti tables for small $n$.

Finally, the free resolution of the apolar ideal $\det_n^{\perp}$ is naturally a free resolution of representations of the symmetry group $G_{\det_n}$ of the determinant $\det_n$.  In \cref{sec:repn-theory}, we provide a complete representation-theoretic description of the spaces of linear generators, relations and second syzygies (\cref{thm:generators-repn,thm:relations-repn,thm:2ndsyzygies-repn}).  An outline of the paper is provided in \cref{sec:outline}.

\subsection{Motivation} \label{sec:motivation}
Understanding the difference between the ``complexity'' of the determinant and permanent polynomials is of fundamental significance to algebraic complexity theory.  There are several notions of complexity of a homogeneous polynomial such as determinantal complexity, Waring rank, and product rank.
Valiant conjectured in \cite{valiant1} and \cite{valiant2} that the determinantal complexity of the $n \times n$ permanent $\perm_n$ is not bounded above by any polynomial in $n$ and, moreover, showed that this conjecture implies a separation between the algebraic complexity classes $\VP_{\rm e}$ and  $\VNP$, which are algebraic analogues of P and NP.

Since the apolar ideal uniquely determines a homogenous polynomial up to scaling, it is natural to ask the following vaguely formulated question:
\begin{question}  \label{ques3}
Can algebraic or homological properties of the ideal $f^{\perp}$ be used to give lower or upper bounds for any complexity measure?
\end{question}

This question was our original motivation and we decided to focus on the minimal free graded resolutions of $f^{\perp}$.  For instance, in \cite{ranestad-schreyer}, it was shown that if $f^{\perp}$ is generated in degree $D$, then the Waring rank of $f$ is bounded below by $\frac{1}{D} \dim_k S/f^{\perp}$.  One might hope that there is a stronger lower bound involving  higher syzygies.   

Our search for new lower bounds of the Waring rank, determinantal complexity, and other complexity measures in terms of the minimal graded free resolutions of $f^{\perp}$ has so far been elusive.   Therefore, we unfortunately have no positive answers to \cref{ques3}.  Nevertheless, we find the problem of determining the syzygies of the apolar ideals of $\det_n$ and $\perm_n$ intrinsically interesting and our work has generating appealing connections to combinatorics and representation theory.  

\subsection*{Acknowledgements} During the preparation of this paper, the first author was partially supported by the Australian Research 
Council grant DE140101519. The second author was partially supported by David Smyth's Australian Research Council grant DE140100259.  Both authors would like to thank Daniel Erman for providing helpful suggestions.

\section{Background}

\subsection{Notation}
As in the introduction, $\mathbf{x} =(x_{i,j})_{1\leq i,j\leq n}$ will denote an $n \times n$ matrix of indeterminates.  We set $R = \bk[x_{i,j}]$, a polynomial ring in $n^2$ variables.  The polynomials $\det_n$ and $\perm_n$ are elements of this ring.  Let $M_n(\bk)$ be the $\bk$-vector space of $n \times n$ matrices with basis $\{ X_{i,j} \}$ dual to $\{x_{i,j}\}$ for $1\leq i,j\leq n$.  Define $S = \Sym^* M_n(\bk) = \bk[X_{i,j}]$. The apolar ideals $\det_n^\perp$ and $\perm_n^\perp$ are ideals in this ring.  

The operation $\star \co S \times R \to R$ defined in \cref{sec:apolarity} gives $R$ the structure of an $S$-module.  Since no variable appears in $\det_n$  with degree greater than one, we have that $X_{i,j} \star \det_n = \frac{\partial}{\partial x_{i,j}} \det_n$ so that we may identify $\det_n^{\perp}$ with the ideal of polynomials $g \in S$ such that  $g \left( \frac{\partial}{\partial x_{1,1}}, \frac{\partial}{\partial x_{1,2}}, \dotsc, \frac{\partial}{\partial x_{n,n}} \right)( \det_n) = 0$. The same applies for $\perm_n$.

The first observation to make is that the action of $\frac{\partial}{\partial x_{i,j}}$ on $\det_n$ is the same as taking the $i,j$th minor, up to sign; that is,
\begin{equation*}
X_{i,j} \star \det_n \mathbf x = \frac{\partial}{\partial x_{i,j}} \det_n \mathbf x = (-1)^{i+j} \det_{n-1} \mathbf x(i; j)
\end{equation*}
where $\mathbf x(i; j)$ denotes the submatrix of $\mathbf x$ obtained by deleting the $i$th row and $j$th column. Similarly,
\begin{equation*}
X_{i,j} \star \perm_n \mathbf x = \frac{\partial}{\partial x_{i,j}} \perm_n \mathbf x = \perm_{n-1} \mathbf x(i; j)
\end{equation*}
Therefore, since $(S / \det_n^\perp)_d$ and $(S / \perm_n^\perp)_d$ are isomorphic to the spaces of $d$th derivatives of $\det_n$ and $\perm_n$ respectively, and since there are $\binom{n}{d}^2$ minors or permanent-minors of order $d$ for an $n \times n$ matrix and they are linearly independent, we obtain the following fact:
\begin{lemma} \label{thm:dim-S/det}
The dimensions of $(S / \det_n^\perp)_d$ and $(S / \perm_n^\perp)_d$ are each $\binom{n}{d}^2$.
\end{lemma}

Now, consider the minimal graded free resolutions
\begin{equation} \label{E:free-det}
 \dotsb \to F_2 \xrightarrow{d_2} F_1 \xrightarrow{d_1} F_0 \to S / \det_n^\perp \to 0
 \end{equation}

and
\begin{equation} \label{E:free-perm}
 \dotsb \to F'_2 \xrightarrow{d'_2} F'_1 \xrightarrow{d'_1} F'_0 \to S / \perm_n^\perp \to 0
 \end{equation}
 where $F_i$ and $F_i'$ are free graded $S$-modules of the form
$\bigoplus_j S(-j)^{\beta_{i,j}}$.
Elements of the kernel of $d_{i+1}$ or $d_{i+1}'$ are called \emph{$i$th syzygies}.  The numbers $\beta_{i,j}$ are called the {\it graded Betti numbers}.  Clearly $F_0 = F_0' = S$. The summands of $F_1$ (resp. $F'_1$) correspond to a minimal set of generators of $\det_n^\perp$ (resp. $\perm_n^{\perp}$).   

\subsection{Generators} \label{sec:shafiei}
Shafiei determined sets of minimal generators for $\det_n^\perp$ and $\perm_n^{\perp}$:

\begin{theorem}{\rm {\cite[Thms 2.12--13]{shafiei}}} \label{thm:Shafiei}  \label{thm:beta12}
The apolar ideal $\det_n^\perp$ is minimally generated by the following polynomials:
\begin{align*}
& X_{i,j}^2, && \text{for $i,j = 1, \dotsc, n$;} \\
& X_{i,j} X_{i,k}, && \text{for $i,j,k = 1, \dotsc, n$, $j \neq k$;} \\
& X_{i,j} X_{k,j}, && \text{for $i,j,k = 1, \dotsc, n$, $i \neq k$; and} \\
& X_{i,j} X_{k,l} + X_{i,l} X_{k,j}, && \text{for $i,j,k,l = 1, \dotsc, n$ and $i \neq k$, $j \neq l$.}
\end{align*}
The apolar ideal $\perm_n^\perp$ is minimally generated by the following polynomials:
\begin{align*}
& X_{i,j}^2, && \text{for $i,j = 1, \dotsc, n$;} \\
& X_{i,j} X_{i,k}, && \text{for $i,j,k = 1, \dotsc, n$, $j \neq k$;} \\
& X_{i,j} X_{k,j}, && \text{for $i,j,k = 1, \dotsc, n$, $i \neq k$; and} \\
& X_{i,j} X_{k,l} - X_{i,l} X_{k,j}, && \text{for $i,j,k,l = 1, \dotsc, n$ and $i \neq k$, $j \neq l$.}
\end{align*}
In particular, both ideals are generated by 
$\beta_{1,2}  = \binom{n+1}{2}^2$
quadrics, and all other graded Betti numbers $\beta_{1,j}$ for $j \neq 2$ are zero.
\end{theorem}

\subsection{Relations}
The main goal of this paper is to describe $F_2$ and $F_2'$, the relations between these generators. Elements in $F_1 = \bigoplus_j S(-j)^{\beta_{2,j}}$ and $F_1'$  may be thought of as formal $S$-linear combinations of the generators of $\det_n^\perp$ or $\perm_n^\perp$, and we will write them as such: e.g.
\begin{equation*}
X_{1,2} (X_{2,1}^2) + X_{1,1} (X_{2,1} X_{2,2}) - X_{2,1} (X_{1,1} X_{2,2} + X_{1,2} X_{2,1})
\end{equation*}
If the $S$-coefficients all have degree $1$, we call the relation \emph{linear}, and if the $S$-coefficients have degree $2$, it is \emph{quadratic}.

\subsection{Gradings} \label{S:gradings}
There are three gradings of $S = \bk[X_{i,j}]$ that will be important in this paper:
\begin{itemize}
	\item {\it standard grading}: This is the grading by $\mathbb Z$ where each $X_{i,j}$ has degree $1$.
	\item {\it multigrading}:  This is the grading by $\mathbb Z^n \times \mathbb Z^n$, where $X_{i,j}$ has degree $e_i + f_j$, $i,j \in \{1, \dotsc, n\}$, where $e_i$ and $f_j$ are the standard basis elements of each copy of $\mathbb Z^n$.
	\item {\it monomial grading}: This is the grading by $\mathbb Z^{n \times n}$ where $X_{i,j}$ has degree $e_{i,j}$.  We will sometimes write monomial degrees in the form of a matrix, in the obvious way, by writing the coefficient of $e_{i,j}$ in position $(i,j)$. 
\end{itemize}
Each of these gradings is strictly finer than the last: any element that is homogeneous with respect to monomial degree is homogeneous with respect to multidegree, and similarly for multidegree and standard degree. For example, the polynomial $X_{1,1}^2 X_{1,2} X_{2,2}$ in $\bk[X_{1,1}, \dotsc, X_{2,2}]$ has standard degree $4$, multidegree $3e_1 + e_2 + 2f_1 + 2f_2 = \big( (3,1), (2,2) \big)$, and monomial degree $2e_{1,1} + e_{1,2} + e_{2,2}$ or equivalently $\begin{bmatrix} 2 & 1 \\ & 1 \end{bmatrix}$ in matrix notation.

We can extend these gradings to $F_1$ and $F_1'$: the degree of an element $f \cdot (g)$, where $f \in S$ and $g$ is a generator of $\det_n^\perp$ or $\perm_n^\perp$, is the sum of the degrees of $f$ and $g$. Note that ``linear'' elements of $F_1$ and $F_1'$ actually have standard degree $2+1 = 3$ in this sense, and ``quadratic'' elements have standard degree $4$, since all generators $g$ of $\det_n^\perp$ and $\perm_n^\perp$ have standard degree $2$.

Our proof of \Cref{main-theorem} is divided into two cases according to the following definition:

\begin{definition} \label{D:singular-plural}
If a multidegree is a tuple consisting only of $0$'s and $1$'s, we will call it \emph{singular}; otherwise, it is \emph{plural}.
\end{definition}

\subsection{Symmetries} \label{sec:symmetries}
The symmetries of the determinant and permanent will play an important role in this paper.  The determinant is invariant (up to scaling) under multiplying the $n \times n$ matrix $\bx$ of indeterminates on the left and right by any two $n \times n$ matrices and under transposing.  It is a theorem of Frobenius \cite{frobenius} that these are all the symmetries.  That is,  if we consider $M_n(\bk) = V \tensor W$ where $V$ and $W$ are $n$-dimensional vectors spaces over $\bk$, then $\GL(M_n(\bk))$ acts on the vector space $R_n = \Sym^n M_n(\bk)^{\vee}$ and the stabilizer of $\det_n$ viewed as an element in the projective space $\PP(R_n)$ is
$$G_{\det_n} = (\GL(V) \times \GL(W))/\bk^* \rtimes \ZZ/2,$$
where $\bk^* \subset \GL(V) \times \GL(W)$ is the subgroup consisting of pairs $(\alpha I_n, \alpha^{-1} I_n)$ for $\alpha \in \bk^*$ (and where $I_n$ denotes the identity matrix).  An element $(A,B)$ in the first factor of $G_{\det_n}$ acts on $M_n(\bk)$ via $M \mapsto AMB^{\top}$ and the non-identity element in the second factor acts via transposition $M \mapsto M^{\top}$.

Similarly, the permanent is invariant (up to scaling) under transposing, permuting the rows and columns, and multiplying the rows and columns by non-zero scalars --- these are all of the symmetries \cite{marcus-may}.  That is, if we let $T_V \subset \GL(V)$ (resp. $T_W \subset \GL(W)$) be the subgroup of diagonal matrices and $N(T_V)$ (resp., $N(T_W)$) be its normalizer, then the stabilizer of $\perm_n \in \PP(R_n)$ is
$$G_{\perm_n} = (N(T_V) \times N(T_W))/\bk^* \rtimes \ZZ/2,$$
where elements act in a similar fashion to the determinant.

The symmetries of transposing and permuting rows and columns will be particularly important to us; we think of the latter of these as a group action of $\bS_n \times \bS_n$.

Observe that Shafiei's lists of generators for $\det_n^\perp$ and $\perm_n^\perp$ (\cref{thm:Shafiei}) are setwise invariant under transposing as well as under permuting rows and columns, and also that every generator is homogeneous with respect to the standard grading and multigrading. These properties are also inherited by the relations.

The symmetries allow us to write the multidegrees and monomial degrees of syzygies more concisely. Since the space of syzygies is symmetric under permuting rows and columns, the order of the entries in the two $n$-tuples comprising a multidegree is somewhat irrelevant. We can therefore think of the multidegree of a syzygy as a pair of partitions of the integer $m$, where $m$ is the syzygy's standard degree.  For instance, the multidegree of $X_{1,1}^2 X_{1,2} X_{2,2}$ corresponds to the pair of partitions $(3+1, \ 2+2)$.

\subsection{Outline of the proof of \cref{main-theorem}}  \label{sec:outline}

In \cref{sec:singular-multidegree}, we show that all relations of $\det_n^{\perp}$ of singular multidegree are generated by linear relations (\cref{thm:linear-generate-all-singular}).  To establish this, we identify relations of a fixed singular multidegree with certain $\bk$-labelings of the Cayley graph of the symmetric group $\bS_m$ and then study the combinatorics of the Cayley graph.  
In \cref{sec:plural-multidegree}, we show that all relations $\det_n^{\perp}$ of plural multidegree are generated by linear relations.  This allows us finish the proof of \cref{main-theorem}\ref{main-theorem-a}.

In \cref{sec:permanent}, we perform the necessary adjustments to \cref{sec:singular-multidegree,sec:plural-multidegree} to characterize the module of relations of $\perm_n^{\perp}$ and thus establishing \cref{main-theorem}\ref{main-theorem-a}.  Unlike for the determinant, both linear and quadratic relations are needed to generate all relations of the permanent.

\section{Singular multidegree case} \label{sec:singular-multidegree}

In this section, we begin our investigation of the space of relations of the minimal quadratic generators of $\det_n^{\perp}$ as listed in \cref{thm:Shafiei} by focusing on relations of singular multidegree.  The main result is \cref{thm:linear-generate-all-singular} which asserts that all relations of singular multidegree are generated by linear relations.  This theorem is established as follows.  First, we identify monomials in $S = \bk[X_{i,j}]$ of standard degree $m$ and of fixed multidegree with permutations in $\bS_m$; see \cref{sec:singular-permuations}.  Using the Cayley graph $\Gamma(\bS_m)$ of the symmetric group $\bS_m$, we identify the space of relations of standard degree $m$ and of fixed multidegree with the space of certain $\bk$-labelings called {\it zero-magic labelings} (see \Cref{D:zero-magic}) on $\Gamma(\bS_m)$ (\cref{lem:relations}).  By a result of Doob (\cref{thm:Doob}), the space of zero-magic labelings is spanned by {\it cycle labelings} (see \cref{D:cycle-labeling}).  Finally, by studying the combinatorics of the Cayley graph, we show that any cycle labeling of $\Gamma(\bS_m)$ is the sum of certain commutator labelings corresponding to linear relations (\cref{thm:commutators-generate-all-cycles}). 

\subsection{Singular multidegree and permutations} \label{sec:singular-permuations}
Recall from \Cref{D:singular-plural} that a multidegree is \emph{singular} if it is a tuple of only 0s and 1s. Firstly, observe that the generators $(X_{i,j}^2)$, $(X_{i,j} X_{i,k})$ and $(X_{i,j} X_{k,j})$ each contribute $2$ to the multidegree in some row or column, so elements with singular multidegree can only involve the generator $(X_{i,j} X_{k,l} + X_{i,l} X_{k,j})$. Secondly, if a monomial in $S$ of standard degree $m$ has singular multidegree $e_{i_1} + \cdots + e_{i_m} + f_{j_1} + \cdots + f_{j_m}$, then the monomial necessarily has the form $X_{i_1, j_{\sigma(1)}} \cdots X_{i_m, j_{\sigma(m)}} $ for some permutation $\sigma \in \bS_m$.  In other words, if we fix a singular multidegree with standard degree $m$, there is a one-to-one correspondence between monomials in $S$ with this multidegree (up to scaling) and elements of the symmetric group $\bS_m$.

\subsection{Cayley graphs}
To progress further with this train of thought, we must first define a certain graph.

\begin{definition}
The \emph{Cayley graph} of a group $A$ together with a set of generators $G$ is the directed graph whose vertex set is $A$, with a directed edge from $a$ to $a g$ for each $a \in A$ and each generator $g \in G$.
\end{definition}

\begin{remark}
Any Cayley graph is connected. Indeed, for vertices $a_1, a_2 \in A$, the element $a_1^{-1} a_2$ must be expressible as $g_1 \dotsm g_r$ for some $g_1, \dotsc, g_r \in G$, since $G$ is a generating set. Therefore $g_1, \dotsc, g_r$ describe a path between $a_1$ and $a_1 (a_1^{-1} a_2) = a_2$.
\end{remark}

The symmetric group $\bS_m$ is generated by the set of transpositions, that is, the permutations of the form $(i \ j)$ for distinct $i,j \in \{1, \dotsc, m\}$. Therefore we may construct the Cayley graph of $\bS_m$ with this generating set. In this case, since every transposition is its own inverse, each directed edge between vertices $a$ and $a'$ given by a transposition $\tau$ has a matching edge from $a'$ to $a$ given by $\tau^{-1} = \tau$, so we will instead consider the simpler \emph{undirected Cayley graph of $\bS_m$} made by merging each of these pairs of edges into a single, undirected edge. Call this graph $\Gamma(\bS_m)$.

\begin{lemma}
The graph $\Gamma(\bS_m)$ is bipartite. 
\end{lemma}

\begin{proof}
The vertices with odd and even sign as permutations form a bipartition, since every edge necessarily connects a vertex with odd sign to one with even sign.
\end{proof}

Let us now link this back to the apolar ideal. For a fixed singular multidegree $\mu$ of standard degree $m$, we already saw that the set of monomials (up to scaling) with multidegree $\mu$ is in bijection with $\bS_m$. Consider an element of the form $f \cdot (X_{i,j} X_{k,l} + X_{i,l} X_{k,j}) \in F_1$ with multidegree $\mu$ where $f \in S$ is a monomial. Let $\sigma_1, \sigma_2 \in \bS_m$ be the permutations corresponding to the monomials $f X_{i,j} X_{k,l} \in S$ and $f X_{i,l} X_{k,j} \in S$ under the above bijection.  
Moreover, the monomial degrees of $f X_{i,j} X_{k,l}$ and $f X_{i,l} X_{k,j}$ differ by a transposition, namely $(j \ l)$.  Thus we associate $f \cdot (X_{i,j} X_{k,l} + X_{i,l} X_{k,j})$ to the edge $(j \ l)$ between  $\sigma_1$ and  $\sigma_2$ in the Cayley graph $\Gamma(\bS_m)$.

To specify an element of $F_1$ of multidegree $\mu$, we must specify only the $S$-coefficients of generators of the form $(X_{i,j} X_{k,l} + X_{i,l} X_{k,j})$ or, equivalently, the $\bk$-coefficients of terms of the form $f \cdot (X_{i,j} X_{k,l} + X_{i,l} X_{k,j})$ where $f \in S$ is a monomial.  This in turn precisely corresponds to a $\bk$-labeling of the edges of $\Gamma(\bS_m)$.  The condition that the element is a relation, i.e., an element of $\ker(F_1 \xrightarrow{d_1} F_0)$, is that for each monomial in $S$, the sum of the coefficients of the terms in $F_1$ involving it in the relation is zero.  In the Cayley graph interpretation, this equivalently means that at every vertex $v$, the sum of the labels of the edges that meet $v$ is zero.  This last condition on a graph is important enough to warrant its own definition:

\begin{definition} \label{D:zero-magic}
A $\bk$-edge-labeling of a graph is called \emph{zero-magic} if for every vertex, the sum of the labels of the edges meeting this vertex is zero.
\end{definition}

We give the set of edge labelings of $\Gamma(\mathbf S_m)$ a $\Bbbk$-vector space structure in the obvious way, by making addition and scalar multiplication act edgewise. In this way, the set of zero-magic labelings forms a subspace.

To summarize the above discussion, we have:

\begin{lemma} \label{lem:relations}
For a singular multidegree $\mu$ of standard degree $m$, the space of relations of multidegree $\mu$ is isomorphic to the vector space of zero-magic labelings of $\Gamma(\bS_m).$ \epf
\end{lemma}

\begin{definition} \label{D:cycle-labeling}
Given a bipartite graph $G$ and a cycle $C \subseteq G$ with a distinguished edge $e$, we define the \emph{cycle labeling} $\Lambda(C, e)$ to be the $\bk$-labeling where every edge outside $C$ is labeled $0$, and the edges along $C$ are given the alternating labels $1$ and $-1$, starting with the label $1$ for $e$. Since $G$ is bipartite, all cycles have even length, so this definition makes sense.  Also, we allow a cycle to travel along the same edge multiple times; if this occurs, the labels of such an edge are added together.
\end{definition}

Every cycle labeling is a zero-magic labeling. However, Doob established a far stronger result:

\begin{theorem}{\rm \cite[Prop.~2.4]{doob}} \label{thm:Doob}
Let $G$ be a connected bipartite graph. The vector space of zero-magic labelings of $G$ is spanned by the cycle labelings.

Moreover, if $T$ is a spanning tree of $G$, adding any single edge from $E(G) \setminus E(T)$ to $T$ must introduce a cycle; if we pick one cycle $C_e \subseteq T \cup \{e\}$ for each edge $e \in E(G) \setminus E(T)$, then the cycle labelings $\Lambda(C_e, e)$ form a basis for the vector space of zero-magic labelings.
\end{theorem}

\begin{remark}  In particular, the dimension of the vector space of zero-magic labelings is $\left| E(G) \setminus E(T) \right|$, which is often called the \emph{circuit rank} of $G$.
\end{remark}

\subsection{Commutator cycles}
\Cref{thm:Doob} allows us to restrict our attention to only those relations of singular multidegree that correspond to cycle labelings in $\Gamma(\bS_m)$. These have a simple description in terms of permutations: a cycle in $\Gamma(\bS_m)$ may be specified by a starting vertex and a sequence of transpositions in $\bS_m$ that compose to give the identity permutation. Since Cayley graphs are clearly vertex-transitive, any sequence of transpositions that compose to the identity may define a cycle at any starting vertex, so we will often ignore the datum of the starting point.

We now define a specific class of cycles in $\Gamma(\bS_m)$.

\begin{definition}
Let $i,j,k \in \{1, \dotsc, m\}$ be distinct. We have the following composition of transpositions:
\begin{equation*}
(i \ j) (i \ k) (i \ j) (j \ k) = (1).
\end{equation*}
This sequence of transpositions defines a length-4 cycle in $\Gamma(\bS_m)$, given a starting vertex. A cycle of this form is called a \emph{commutator cycle}, and a cycle labeling built on this cycle is called a \emph{commutator labeling}. 
\end{definition}

\begin{remark}
Note that a length-$4$ cycle in $\Gamma(\bS_m)$ has a choice of four starting points and two directions, so
\begin{equation*} \label{eq:commutator-cycle-rotated}
(i \ k) (i \ j) (j \ k) (i \ j) = (1)
\end{equation*}
is also a commutator cycle. (The other six choices of starting point and direction can be made from these two sequences by interchanging $i$, $j$ and $k$.)
More abstractly, commutator cycles can be written
\begin{equation*}
a b a [aba] \quad \text{or} \quad a b [bab] b
\end{equation*}
for transpositions $a$ and $b$ that are distinct but not disjoint, where $[aba]$ means the single transposition that is $b$ conjugated by $a$.
\end{remark}

\begin{remark}
The prototypical commutator cycle is the cycle specified by the sequence
\begin{equation*}
(1 \ 2) (1 \ 3) (1 \ 2) (2 \ 3)
\end{equation*}
starting at the vertex $(1)$. In the correspondence between zero-magic labelings and relations, this corresponds to the linear relation
\begin{multline} \label{E:rhoS}
\rho_S = X_{3,3} (X_{1,1} X_{2,2} + X_{1,2} X_{2,1}) - X_{1,2} (X_{2,1} X_{3,3} + X_{2,3} X_{3,1}) \\ {} + X_{2,3} (X_{1,1} X_{3,2} + X_{1,2} X_{3,1}) - X_{1,1} (X_{2,2} X_{3,3} + X_{2,3} X_{3,2})
\end{multline}
when $m = 3$, and a monomial times this when $m > 3$. Note that for a fixed $n$ and singular multidegree, any commutator cycle of standard degree $m$ can be obtained from this cycle by renaming the indices, that is, by permuting rows and columns.
\end{remark}

\begin{proposition}
All cycle labelings on cycles of length $4$ in $\Gamma(\bS_m)$ are sums of commutator labelings, or zero.
\end{proposition}

\begin{proof}
A cycle of length $4$ corresponds to a sequence of $4$ transpositions $a b c d$ that composes to the identity. This means that $cd = (ab)^{-1}$.  There are three possibilities for $a$ and $b$:
\begin{itemize}[leftmargin=*]
\item $a = b$. In this case we must have $c = d$, so the cycle labeling gives each of the edges $a = b$ and $c = d$ the label $1 - 1 = 0$, so this is the zero labeling.

\item $a$ and $b$ are distinct but not disjoint. Then $a = (i \ j)$, $b = (i \ k)$, for some $i,j,k$. Thus $ab = (i \ k \ j)$, so $cd = (i \ j \ k)$; this means one of the following:
\begin{itemize}
\item $c = (i \ j)$ and $d = (j \ k)$. Then $abcd = (i \ j) (i \ k) (i \ j) (j \ k)$, a commutator cycle of the first type.
\item $c = (j \ k)$ and $d = (i \ k)$. Then $abcd = (i \ j) (i \ k) (j \ k) (i \ k)$, which is a commutator cycle of the second type.
\item $c = (i \ k)$ and $d = (i \ j)$. Then $abcd = (i \ j) (i \ k) (i \ k) (i \ j)$, so the cycle goes along the same two edges forwards then backwards. Thus the labeling cancels to zero on both edges, so it is the zero labeling.
\end{itemize}

\item $a$ and $b$ are disjoint. Say $a = (i \ j)$ and $b = (k \ l)$. Then either $c = (k \ l)$ and $d = (i \ j)$, in which case the cycle goes along the same edges forwards and backwards and thus cancels to zero; or $c = (i \ j)$ and $d = (k \ l)$, so $abcd = (i \ j) (k \ l) (i \ j) (k \ l)$. Consider the following subgraph of $\Gamma(\bS_m)$:
$$\xymatrix{
(1) \ar@{-}[rrr]^{(i \ j)} \ar@{-}[ddd]_{(k \ l)}  \ar@{-}[rd]_{(j \ k)}	& 				&				& (i \ j)  \ar@{-}[ddd]^{(k \ l)}  \ar@{-}[ld]^{(i \ k)} \\
					& (j \ k) \ar@{-}[r]^{(i \ j)} 	\ar@{-}[d]^{(j \ l)} 		& (i \ k \ j)	\ar@{-}[d]^{(i\ l)} 	\\
					& (j \ l \ k)	\ar@{-}[r]^{(i \ j)} 		& (i \ l \ k \ j)	\\	
(k \ l) \ar@{-}[rrr]_{(i \ j)} \ar@{-}[ru]^{(j \ k)} 	& 				&				& (i \ j)(k \ l) \ar@{-}[ul]_{(i \ k)} 
}$$

The outer square is $abcd$, but each of the five inner squares is a commutator cycle. By choosing signs appropriately, we may add together the labelings given by these five inner commutator cycles so that the labels on the internal edges cancel, leaving only the cycle labeling of the outer square. Hence the cycle labeling for $(i \ j) (k \ l) (i \ j) (k \ l)$ is a sum of commutator labelings. (We assumed that $abcd$ started at the vertex $(1)$ in this diagram, but vertex-transitivity means that we equally could have started at any vertex.) \qedhere
\end{itemize}
\end{proof}

\subsection{Commutation rules}
A consequence of this proposition is that we obtain commutation rules for transpositions. Suppose we have some loop containing adjacent edges $a$ and $b$. If $a = b$, the labels on this edge sum to zero, so we can effectively cancel these two transpositions with each other. If $a$ and $b$ are distinct but not disjoint, then we may commute the edges following the rule $ab = [aba] a = b [bab]$ in the cycle, by adding the commutator labelings given by the cycles $aba[aba]$ or $ab[bab]b$, with the right choice of sign:
\begin{equation*}
\begin{tikzpicture}[vertex/.style={circle, inner sep=1.8pt, fill}, auto, on grid]
\node (left dots) {$\dots$};
\node [vertex] (sigma) [right=of left dots] {};
\node [vertex] (a sigma) [above right=of sigma] {};
\node [vertex] (ab sigma) [below right=of a sigma] {};
\node [vertex] (aba sigma) [below right=of sigma] {};
\node (right dots) [right=of ab sigma] {$\dots$};
\draw (left dots) to (sigma);
\draw (sigma) to node {$a$} (a sigma) to node {$b$} (ab sigma);
\draw (ab sigma) to (right dots);
\draw [dashed] (sigma) to node [swap] {$aba$} (aba sigma) to node [swap] {$a$} (ab sigma);
\node [below=of a sigma] {$\downarrow$};
\end{tikzpicture}
\qquad
\begin{tikzpicture}[vertex/.style={circle, inner sep=1.8pt, fill}, auto, on grid]
\node (left dots) {$\dots$};
\node [vertex] (sigma) [right=of left dots] {};
\node [vertex] (a sigma) [above right=of sigma] {};
\node [vertex] (ab sigma) [below right=of a sigma] {};
\node [vertex] (b sigma) [below right=of sigma] {};
\node (right dots) [right=of ab sigma] {$\dots$};
\draw (left dots) to (sigma);
\draw (sigma) to node {$a$} (a sigma) to node {$b$} (ab sigma);
\draw (ab sigma) to (right dots);
\draw [dashed] (sigma) to node [swap] {$b$} (b sigma) to node [swap] {$bab$} (ab sigma);
\node [below=of a sigma] {$\downarrow$};
\end{tikzpicture}
\end{equation*}

And if $a$ and $b$ are disjoint, then we can commute $ab = ba$ by adding on the labeling of the length-$4$ cycle $abab$, which we just saw could be built from commutator cycles:
\begin{equation*}
\begin{tikzpicture}[vertex/.style={circle, inner sep=1.8pt, fill}, auto, on grid]
\node (left dots) {$\dots$};
\node [vertex] (sigma) [right=of left dots] {};
\node [vertex] (a sigma) [above right=of sigma] {};
\node [vertex] (ab sigma) [below right=of a sigma] {};
\node [vertex] (b sigma) [below right=of sigma] {};
\node (right dots) [right=of ab sigma] {$\dots$};
\draw (left dots) to (sigma);
\draw (sigma) to node {$a$} (a sigma) to node {$b$} (ab sigma);
\draw (ab sigma) to (right dots);
\draw [dashed] (sigma) to node [swap] {$b$} (b sigma) to node [swap] {$a$} (ab sigma);
\node [below=of a sigma] {$\downarrow$};
\end{tikzpicture}
\end{equation*}

We finally reach the following proposition:

\begin{proposition} \label{thm:commutators-generate-all-cycles}
Any cycle labeling in $\Gamma(\bS_m)$ is the sum of commutator labelings, or zero.
\end{proposition}

\begin{proof}
We need to show that any sequence of transpositions in $\bS_m$ that composes to the identity can be reduced to the identity using the commutation rules introduced above.
We proceed by induction on $m$. Observe that when $m = 2$, $\Gamma(\bS_2)$ only has one edge, between $(1)$ and $(1 \ 2)$, so all cycles simply travel along this edge backwards and forwards, and all cycle labelings cancel to $0$.

Now, let $m > 2$. Suppose we have some sequence
\begin{equation*}
\tau_1 \tau_2 \dotsm \tau_s = (1)
\end{equation*}
of transpositions in $\bS_m$. We describe how to reduce this to a sequence whose transpositions are all contained in the subgroup isomorphic to $\bS_{m-1}$ of permutations that fix $m$.
Every transposition either fixes $m$ or moves it. If all transpositions in our sequence fix $m$, we are done, so assume that some transpositions move $m$, and consider the left-most one of these: suppose it is $\tau_r = (i \ m)$. We want to commute this towards the right. If $\tau_r$ and $\tau_{r+1}$ are disjoint, use the rule
\begin{equation*}
(i \ m) (j \ k) = (j \ k) (i \ m).
\end{equation*}
If $\tau_r$ and $\tau_{r+1}$ are distinct but not disjoint and $\tau_{r+1}$ fixes $m$, use the rule
\begin{equation*}
(i \ m) (i \ j) = (i \ j) (j \ m).
\end{equation*}
If $\tau_r$ and $\tau_{r+1}$ are distinct but not disjoint and $\tau_{r+1}$ moves $m$, use the rule
\begin{equation*}
(i \ m) (j \ m) = (i \ j) (i \ m).
\end{equation*}
If $\tau_r$ and $\tau_{r+1}$ are equal, they cancel:
\begin{equation*}
(i \ m) (i \ m) = (1).
\end{equation*}
After applying any of these rules, the left-most $m$-moving transposition is strictly closer to the right-hand end of the sequence. Therefore if we repeat this process, we may continue for no more than $s-1$ steps, until either there are no more $m$-moving transpositions, or there is a single one and it is at the very right-hand end of the string. But this latter case is impossible: a sequence of the form
\begin{align*}
\tau_1 \dotsm \tau_{s'-1} \tau_{s'} = (1)
\end{align*}
where $\tau_1, \ldots, \tau_{s'-1}$ fix $m$, but $\tau_{s'}$ moves $m$, must as a whole move $m$, so it cannot equal $(1)$, and we have a contradiction. Therefore our algorithm must produce a sequence of transpositions where none involve $m$; that is, a string contained in $\bS_{m-1}$. But by the induction hypothesis, every loop in $\bS_{m-1}$ can be reduced to the identity by these commutator rules.
\end{proof}

\begin{theorem} \label{thm:linear-generate-all-singular}
All relations of singular multidegree are generated by linear relations and specifically by the orbit of $\rho_S$ (introduced in \eqref{E:rhoS}), under the symmetry action of permuting rows and columns. 
\end{theorem}

\begin{proof}
By \cref{lem:relations}, the vector space of relations of a fixed singular multidegree is isomorphic to the space of zero-magic labelings of $\Gamma(\bS_m)$. By \cref{thm:Doob}, the latter space is spanned by the cycle labelings, and by \cref{thm:commutators-generate-all-cycles}, every cycle labeling is a sum of commutator labelings. But every commutator labeling is an $S$-multiple of an element in the orbit of $\rho_S$ under $\bS_n \times \bS_n$.
\end{proof}

\begin{corollary} \label{cor:dimension-linear-relations}
The vector space of linear relations of a fixed singular multidegree has dimension $4$.
\end{corollary}

\begin{proof}
By \cref{lem:relations}, the vector space of linear relations of a fixed singular multidegree is isomorphic to the space of zero-magic labelings of $\Gamma(\bS_3)$. By \cref{thm:Doob}, the dimension of this space is the circuit rank of $\Gamma(\bS_3)$. But $\Gamma(\bS_3)$ is isomorphic to the complete bipartite graph $K_{3,3}$: it is a bipartite graph with $\left| \bS_3 \right| = 6$ vertices, and each vertex has degree $3$, since there are $\binom{3}{2} = 3$ transpositions in $\bS_3$. The circuit rank of $K_{3,3}$ is $4$. (The graph $\Gamma(\bS_3)$ is shown in \cref{fig:Gamma(S_3)}, with a spanning tree highlighted.)
\end{proof}

\begin{figure}
\caption{The graph $\Gamma(\bS_3)$. The thicker edges form a spanning tree.} \label{fig:Gamma(S_3)}
\centering
\includegraphics[keepaspectratio=true]{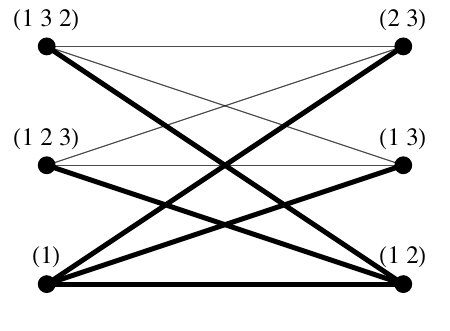}
\end{figure}

\section{Plural multidegree case} \label{sec:plural-multidegree}

In this section, we prove \cref{main-theorem}\ref{main-theorem-a} asserting that all relations of the apolar ideal $\det_n^{\perp}$ are generated by linear relations as long as $\chartext{\bk} \neq 2$.  Moreover, in \cref{thm:linear-relations-list}, a minimal generating set of linear relations is given.

In this section, we find it useful to describe relations using a `dots-and-boxes' notation as detailed in \cref{sec:dots-and-boxes}.  After determining a basis of the linear relations in \cref{thm:linear-relations-list}, we turn our attention to relations of standard degree $\ge 2$ and plural multidegree, since relations of singular multidegree were fully investigated in the previous section.  We first show that quadratic relations of plural multidegree are generated by linear relations (\cref{thm:linear-relations-generate-quadratic}) and then establish that all relations of plural multidegree are generated by linear relations (\cref{thm:linear-generate-all-plural}).

\subsection{Dots-and-boxes notation} \label{sec:dots-and-boxes}

It will be helpful to develop a shorthand notation to express the relations among the generators of $\det_n^{\perp}$ listed in  \cref{thm:Shafiei}. Recall the free resolution of $S/\det_n^{\perp}$ from \eqref{E:free-det} and that elements of $F_1$ are formal $S$-linear sums of the generators of $\det_n^\perp$. We can display this information pictorially in a matrix by showing the generator as a rectangle whose corners are at the locations of the variables that comprise it, and denoting monomials in $S$ with dots in the positions corresponding to the variables (or numbers instead of dots, to indicate multiplicity greater than $1$). We will only display the minimal submatrix in which all the variables appear. For example:
\begin{align*}
X_{2,1} (X_{1,2} X_{2,2}) & = \begin{relmatrix}{2}{2} \relgen{1}{2}{2}{2} \relcoeff{2}{1} \end{relmatrix} \\
X_{1,2} X_{3,3} (X_{1,1} X_{2,2} + X_{1,2} X_{2,1}) & = \begin{relmatrix}{3}{3} \relgen{1}{1}{2}{2} \relcoeff{1}{2} \relcoeff{3}{3} \end{relmatrix} \\
X_{2,2} X_{2,3} (X_{1,1}^2) - X_{1,1}^2 (X_{2,2} X_{2,3}) & = \begin{relmatrix}{2}{3} \relgen{1}{1}{1}{1} \relcoeff{2}{2} \relcoeff{2}{3} \end{relmatrix} - \begin{relmatrix}{2}{3} \relgen{2}{2}{2}{3} \relcoeff[2]{1}{1} \end{relmatrix}.
\end{align*}

\subsection{Linear relations}

\begin{proposition} \label{thm:linear-relations-list}
The space of linear relations is generated by the orbit of the following six relations under the symmetry of permuting rows and columns and transposing:
\begin{equation*} \label{E:linear-det}
\begin{aligned}
\rho_1 &= X_{1,2} (X_{1,1}^2) - X_{1,1} (X_{1,1} X_{1,2}) \\
\rho_2 &= X_{1,3} (X_{1,1} X_{1,2}) - X_{1,2} (X_{1,1} X_{1,3}) \\
\rho_3 &= X_{2,1} (X_{1,1} X_{1,2}) - X_{1,2} (X_{1,1} X_{2,1}) \\
\rho_4 &= X_{1,1} (X_{1,1} X_{2,2} + X_{1,2} X_{2,1}) - X_{2,1} (X_{1,1} X_{1,2}) - X_{2,2} (X_{1,1}^2) \\
\rho_5 &= X_{1,3} (X_{1,1} X_{2,2} + X_{1,2} X_{2,1}) - X_{2,2} (X_{1,1} X_{1,3}) - X_{2,1} (X_{1,2} X_{1,3}) \\
\rho_S &= \begin{multlined}[t] X_{3,3} (X_{1,1} X_{2,2} + X_{1,2} X_{2,1}) - X_{1,2} (X_{2,1} X_{3,3} + X_{2,3} X_{3,1}) \\[2.5pt] {} + X_{2,3} (X_{1,1} X_{3,2} + X_{1,2} X_{3,1}) - X_{1,1} (X_{2,2} X_{3,3} + X_{2,3} X_{3,2}).
\end{multlined}
\end{aligned}
\end{equation*}
\end{proposition}
\begin{remark}
In dots-and-boxes notation, these relations are
$$\begin{aligned}
\rho_1 &= \begin{relmatrix}{1}{2} \relgen{1}{1}{1}{1} \relcoeff{1}{2} \end{relmatrix} - \begin{relmatrix}{1}{2} \relgen{1}{1}{1}{2} \relcoeff{1}{1} \end{relmatrix} \\
\rho_2 &= \begin{relmatrix}{1}{3} \relgen{1}{1}{1}{2} \relcoeff{1}{3} \end{relmatrix} - \begin{relmatrix}{1}{3} \relgen{1}{1}{1}{3} \relcoeff{1}{2} \end{relmatrix} \\
\rho_3 &= \begin{relmatrix}{2}{2} \relgen{1}{1}{1}{2} \relcoeff{2}{1} \end{relmatrix} - \begin{relmatrix}{2}{2} \relgen{1}{1}{2}{1} \relcoeff{1}{2} \end{relmatrix} \\
\rho_4 &= \begin{relmatrix}{2}{2} \relgen{1}{1}{2}{2} \relcoeff{1}{1} \end{relmatrix} - \begin{relmatrix}{2}{2} \relgen{1}{1}{1}{2} \relcoeff{2}{1} \end{relmatrix} - \begin{relmatrix}{2}{2} \relgen{1}{1}{1}{1} \relcoeff{2}{2} \end{relmatrix} \\
\rho_5 &= \begin{relmatrix}{2}{3} \relgen{1}{1}{2}{2} \relcoeff{1}{3} \end{relmatrix} - \begin{relmatrix}{2}{3} \relgen{1}{1}{1}{3} \relcoeff{2}{2} \end{relmatrix} - \begin{relmatrix}{2}{3} \relgen{1}{2}{1}{3} \relcoeff{2}{1} \end{relmatrix} \\
\rho_S &= \begin{relmatrix}{3}{3} \relgen{1}{1}{2}{2} \relcoeff{3}{3} \end{relmatrix} - \begin{relmatrix}{3}{3} \relgen{2}{1}{3}{3} \relcoeff{1}{2} \end{relmatrix} + \begin{relmatrix}{3}{3} \relgen{1}{1}{3}{2} \relcoeff{2}{3} \end{relmatrix} - \begin{relmatrix}{3}{3} \relgen{2}{2}{3}{3} \relcoeff{1}{1} \end{relmatrix}.
\end{aligned}$$
\end{remark}

\begin{proof}
To list the linear relations, we will split $F_1$ into multidegree-homogeneous components. Recall from \cref{sec:symmetries} that we may specify a degree $m$ multidegree up to symmetry by giving a pair of partitions of $m$. Linear relations have $m = 3$, and there are $3$ partitions of $3$ (namely $3$, $2+1$ and $1+1+1$); therefore, there are $3^2 = 9$ multidegrees to consider:
\begin{equation*}
\begin{matrix}
(3, \ 3) & (3, \ 2+1) & (3, \ 1+1+1) \\
(2+1, \ 3) & (2+1, \ 2+1) & (2+1, \ 1+1+1) \\
(1+1+1, \ 3) & (1+1+1, \ 2+1) & (1+1+1, \ 1+1+1)
\end{matrix}
\end{equation*}
The multidegree $(1+1+1, \ 1+1+1)$ is singular --- we have seen that that the space of relations of this multidegree are generated by the orbit of $\rho_S$ (\cref{thm:linear-generate-all-singular}) and its dimension is $4$ (\cref{cor:dimension-linear-relations}). 
Furthermore, three of the multidegrees above can be obtained from the others by transposing. This leaves five cases left to consider, up to symmetry:
\begin{equation*}
\begin{matrix}
(3, \ 3) & (3, \ 2+1) & (3, \ 1+1+1) \\
\mathmakebox[\widthof{$(2+1, \ 3)$}][c]{-} & (2+1, \ 2+1) & (2+1, \ 1+1+1) \\
\mathmakebox[\widthof{$(1+1+1, \ 3)$}][c]{-} & \mathmakebox[\widthof{$(1+1+1, \ 2+1)$}][c]{-} & \mathmakebox[\widthof{$(1+1+1, \ 1+1+1)$}][c]{-}
\end{matrix}
\end{equation*}

We will identify a basis for the vector space of linear relations by considering these multidegrees case by case. For concreteness, we will write the variables as, say, $X_{1,2}$ instead of $X_{i,j}$, and generalize by symmetry.
\begin{itemize} [leftmargin=*]
\item $(3, \ 3)$. The only term with this multidegree (up to symmetry) is $X_{1,1} (X_{1,1}^2)$, which is the very compact picture $\begin{relmatrix}{1}{1} \relgen{1}{1}{1}{1} \relcoeff{1}{1} \end{relmatrix}$ in the dots-and-boxes notation. Since there is only one possible term, and its image in $F_0$ is the non-zero monomial $X_{1,1}^3$, a relation of this multidegree must be the zero relation.

\item $(3, \ 2+1)$. There are two possible terms with this multidegree (up to symmetry): $X_{1,2} (X_{1,1}^2) = \begin{relmatrix}{1}{2} \relgen{1}{1}{1}{1} \relcoeff{1}{2} \end{relmatrix}$ and $X_{1,1} (X_{1,1} X_{1,2}) = \begin{relmatrix}{1}{2} \relgen{1}{1}{1}{2} \relcoeff{1}{1} \end{relmatrix}$. Both of these get mapped to $X_{1,1}^2 X_{1,2}$ in $F_0$. Therefore all relations of this multidegree are a scalar multiple of the relation
\begin{equation*}
\rho_1 = X_{1,2} (X_{1,1}^2) - X_{1,1} (X_{1,1} X_{1,2}) = \begin{relmatrix}{1}{2} \relgen{1}{1}{1}{1} \relcoeff{1}{2} \end{relmatrix} - \begin{relmatrix}{1}{2} \relgen{1}{1}{1}{2} \relcoeff{1}{1} \end{relmatrix}.
\end{equation*}
It will be useful later to know the dimensions of these spaces of relations, so note that this relation spans a vector space with dimension $1$.

\item $(3, \ 1+1+1)$. There are three terms with this multidegree:
\begin{align*}
X_{1,3} (X_{1,1} X_{1,2}) & = \begin{relmatrix}{1}{3} \relgen{1}{1}{1}{2} \relcoeff{1}{3} \end{relmatrix}, \\
X_{1,2} (X_{1,1} X_{1,3}) & = \begin{relmatrix}{1}{3} \relgen{1}{1}{1}{3} \relcoeff{1}{2} \end{relmatrix}, \text{ and} \\
X_{1,1} (X_{1,2} X_{1,3}) & = \begin{relmatrix}{1}{3} \relgen{1}{2}{1}{3} \relcoeff{1}{1} \end{relmatrix}.
\end{align*}
These all map to $X_{1,1} X_{1,2} X_{1,3} \in F_0$. The relations among them are spanned by
\begin{align*}
\rho_2 = X_{1,3} (X_{1,1} X_{1,2}) - X_{1,2} (X_{1,1} X_{1,3}) & = \begin{relmatrix}{1}{3} \relgen{1}{1}{1}{2} \relcoeff{1}{3} \end{relmatrix} - \begin{relmatrix}{1}{3} \relgen{1}{1}{1}{3} \relcoeff{1}{2} \end{relmatrix}, \\
X_{1,2} (X_{1,1} X_{1,3}) - X_{1,1} (X_{1,2} X_{1,3}) & = \begin{relmatrix}{1}{3} \relgen{1}{1}{1}{3} \relcoeff{1}{2} \end{relmatrix} - \begin{relmatrix}{1}{3} \relgen{1}{2}{1}{3} \relcoeff{1}{1} \end{relmatrix}, \text{ and} \\
X_{1,1} (X_{1,2} X_{1,3}) - X_{1,3} (X_{1,1} X_{1,2}) & = \begin{relmatrix}{1}{3} \relgen{1}{2}{1}{3} \relcoeff{1}{1} \end{relmatrix} - \begin{relmatrix}{1}{3} \relgen{1}{1}{1}{2} \relcoeff{1}{3} \end{relmatrix}.
\end{align*}
Note that the second and third relations are permutations of the first, $\rho_2$.

These relations are linearly dependent since they sum to $0$, but no single one of them generates the entire space by itself, so this space of relations has dimension $2$.

\item $(2+1, \ 2+1)$. There are four terms with this multidegree:
\begin{align*}
X_{2,2} (X_{1,1}^2) & = \begin{relmatrix}{2}{2} \relgen{1}{1}{1}{1} \relcoeff{2}{2} \end{relmatrix}, \\
X_{2,1} (X_{1,1} X_{1,2}) & = \begin{relmatrix}{2}{2} \relgen{1}{1}{1}{2} \relcoeff{2}{1} \end{relmatrix}, \\
X_{1,2} (X_{1,1} X_{2,1}) & = \begin{relmatrix}{2}{2} \relgen{1}{1}{2}{1} \relcoeff{1}{2} \end{relmatrix}, \\
X_{1,1} (X_{1,1} X_{2,2} + X_{1,2} X_{2,1}) & = \begin{relmatrix}{2}{2} \relgen{1}{1}{2}{2} \relcoeff{1}{1} \end{relmatrix}.
\end{align*}
Note that the first maps to $X_{1,1}^2 X_{2,2}$ in $F_0$, the second and third map to $X_{1,1} X_{1,2} X_{2,1}$, and the fourth maps to the sum of these two polynomials. Thus for a linear combination of these to be a relation, if the $\bk$-coefficient of $X_{1,1} (X_{1,1} X_{2,2} + X_{1,2} X_{2,1})$ is $c$, the coefficient of $X_{2,2} (X_{1,1}^2)$ must be $-c$, and the coefficients of $X_{2,1} (X_{1,1} X_{1,2})$ and $X_{1,2} (X_{1,1} X_{2,1})$ must sum to $-c$. Therefore the relations of this multidegree are spanned by
\begin{gather*}
\rho_3 = X_{2,1} (X_{1,1} X_{1,2}) - X_{1,2} (X_{1,1} X_{2,1}) = \begin{relmatrix}{2}{2} \relgen{1}{1}{1}{2} \relcoeff{2}{1} \end{relmatrix} - \begin{relmatrix}{2}{2} \relgen{1}{1}{2}{1} \relcoeff{1}{2} \end{relmatrix} \\
\intertext{and}
\begin{multlined}
\rho_4 = X_{1,1} (X_{1,1} X_{2,2} + X_{1,2} X_{2,1}) - X_{2,1} (X_{1,1} X_{1,2}) - X_{2,2} (X_{1,1}^2) \\[6pt] {} = \begin{relmatrix}{2}{2} \relgen{1}{1}{2}{2} \relcoeff{1}{1} \end{relmatrix} - \begin{relmatrix}{2}{2} \relgen{1}{1}{1}{2} \relcoeff{2}{1} \end{relmatrix} - \begin{relmatrix}{2}{2} \relgen{1}{1}{1}{1} \relcoeff{2}{2} \end{relmatrix}.
\end{multlined}
\end{gather*}
The space of relations of this multidegree has dimension $2$.

\item $(2+1, \ 1+1+1)$. There are six terms with this multidegree:
\begin{align*}
X_{2,3} (X_{1,1} X_{1,2}) & = \begin{relmatrix}{2}{3} \relgen{1}{1}{1}{2} \relcoeff{2}{3} \end{relmatrix} \mapsto X_{1,1} X_{1,2} X_{2,3} = A, \\
X_{2,2} (X_{1,1} X_{1,3}) & = \begin{relmatrix}{2}{3} \relgen{1}{1}{1}{3} \relcoeff{2}{2} \end{relmatrix} \mapsto X_{1,1} X_{1,3} X_{2,2} = B, \\
X_{2,1} (X_{1,2} X_{1,3}) & = \begin{relmatrix}{2}{3} \relgen{1}{2}{1}{3} \relcoeff{2}{1} \end{relmatrix} \mapsto X_{1,2} X_{1,3} X_{2,1} = C, \\
X_{1,3} (X_{1,1} X_{2,2} + X_{1,2} X_{2,1}) & = \begin{relmatrix}{2}{3} \relgen{1}{1}{2}{2} \relcoeff{1}{3} \end{relmatrix} \mapsto B + C, \\
X_{1,2} (X_{1,1} X_{2,3} + X_{1,3} X_{2,1}) & = \begin{relmatrix}{2}{3} \relgen{1}{1}{2}{3} \relcoeff{1}{2} \end{relmatrix} \mapsto A + C, \\
X_{1,1} (X_{1,2} X_{2,3} + X_{1,3} X_{2,2}) & = \begin{relmatrix}{2}{3} \relgen{1}{2}{2}{3} \relcoeff{1}{1} \end{relmatrix} \mapsto A + B.
\end{align*}
If the $\bk$-coefficient of the $i$th term is $c_i$ in a relation, then by comparing terms in $F_0$, we have that $c_1+c_5+c_6=0$, $c_2+c_4+c_6=0$, and $c_3+c_4+c_5=0$.  We see that the relations are spanned by
\begin{gather*}
\begin{multlined}
\rho_5 = X_{1,3} (X_{1,1} X_{2,2} + X_{1,2} X_{2,1}) - X_{2,2} (X_{1,1} X_{1,3}) - X_{2,1} (X_{1,2} X_{1,3}) \\[6pt] {} = \begin{relmatrix}{2}{3} \relgen{1}{1}{2}{2} \relcoeff{1}{3} \end{relmatrix} - \begin{relmatrix}{2}{3} \relgen{1}{1}{1}{3} \relcoeff{2}{2} \end{relmatrix} - \begin{relmatrix}{2}{3} \relgen{1}{2}{1}{3} \relcoeff{2}{1} \end{relmatrix},
\end{multlined} \\
\begin{multlined}
\hphantom{\rho_5 = {}} X_{1,2} (X_{1,1} X_{2,3} + X_{1,3} X_{2,1}) - X_{2,3} (X_{1,1} X_{1,2}) - X_{2,1} (X_{1,2} X_{1,3}) \\[6pt] {} = \begin{relmatrix}{2}{3} \relgen{1}{1}{2}{3} \relcoeff{1}{2} \end{relmatrix} - \begin{relmatrix}{2}{3} \relgen{1}{1}{1}{2} \relcoeff{2}{3} \end{relmatrix} - \begin{relmatrix}{2}{3} \relgen{1}{2}{1}{3} \relcoeff{2}{1} \end{relmatrix},
\end{multlined} \\
\begin{multlined}
\hphantom{\rho_5 = {}} X_{1,1} (X_{1,2} X_{2,3} + X_{1,3} X_{2,2}) - X_{2,3} (X_{1,1} X_{1,2}) - X_{2,2} (X_{1,1} X_{1,3}) \\[6pt] {} = \begin{relmatrix}{2}{3} \relgen{1}{2}{2}{3} \relcoeff{1}{1} \end{relmatrix} - \begin{relmatrix}{2}{3} \relgen{1}{1}{1}{2} \relcoeff{2}{3} \end{relmatrix} - \begin{relmatrix}{2}{3} \relgen{1}{1}{1}{3} \relcoeff{2}{2} \end{relmatrix}.
\end{multlined}
\end{gather*}
The second and third of these are permutations of $\rho_5$. The space of relations of this multidegree has dimension $3$.
\end{itemize}

This covers all multidegrees possible for a linear relation up to symmetry, so the set of permutations and transposes of $\rho_1, \dotsc, \rho_5, \rho_S$ generates all linear relations.
\end{proof}

\begin{proposition} \label{thm:beta23}
The dimension of the vector space of linear relations is
\begin{equation*}
\beta_{2,3} = 4 \binom{n+1}{3} \binom{n+2}{3} = \frac{1}{9} n^2 \left( n+1 \right)^2 \left( n-1 \right) \left( n+2 \right).
\end{equation*}
\end{proposition}

\begin{proof}
The space of all linear relations is the direct sum of the spaces of linear relations of fixed multidegree. We calculated the dimension of these spaces for each multidegree in \cref{cor:dimension-linear-relations} and the proof of \cref{thm:linear-relations-list}, so it only remains to add these together, with multiplicity. These dimensions are reprinted in \cref{tab:dimensions-of-linear-relation-spaces} for reference.

\begin{table}[tp]
\caption{Table of dimensions of spaces of linear relations of fixed multidegree, computed in \cref{thm:linear-relations-list}. The row and column headers are the corresponding partition, and the number of elements of $\mathbb N^n$ inducing such partitions.} \label{tab:dimensions-of-linear-relation-spaces}
\begin{equation*}
\begin{array}{r r | c c c}
&& 3 & 2+1 & 1+1+1 \\
&& n & n(n-1) & \binom{n}{3} \\ \hline
3 & n & 0 & 1 & 2 \\
2+1 & n(n-1) & 1 & 2 & 3 \\
1+1+1 & \binom{n}{3} & 2 & 3 & 4
\end{array}
\end{equation*}
\end{table}

Given an $n \times n$ matrix, we therefore need to know how many ways there are of picking a multidegree corresponding to each partition pair $(p, p')$, for $p,p' \in \{3, \ {2+1}, \ {1+1+1}\}$. Note that we may allocate the rows and columns independently, so we just need to know how many $n$-tuples of non-negative integers have non-zero entries corresponding to each partition $p$. There are $\binom{n}{1} = n$ ways to get the partition $3$, one for each tuple $(0, \dotsc, 0, 3, 0, \dotsc, 0)$; there are $\binom{n}{1} \binom{n-1}{1} = n (n-1)$ ways to get $2+1$, first picking a position for the $2$ and then a position for the $1$; and there are $\binom{n}{3}$ ways to get $1+1+1$.

Therefore the dimension of the space of all linear relations is
\begin{multline*}
0 n^2 + 1 n^2 (n-1) + 2 n \binom{n}{3} 
+ 1 n^2 (n-1) + 2 n^2 (n-1)^2 + 3 n (n-1) \binom{n}{3} \\ {} + 2 n \binom{n}{3} + 3 n (n-1) \binom{n}{3} + 4 \binom{n}{3}^2 
= \frac{1}{9} n^2 (n+1)^2 (n-1) (n+2). \qedhere
\end{multline*}
\end{proof}

\subsection{Relations of higher degree}

\renewcommand{\relscale}{0.76}

Let $L \subset F_1$ be the submodule generated by the linear relations; by \cref{thm:linear-relations-list}, we know that $L$ is generated by the permutations and transposes of $\rho_1, \dots, \rho_5, \rho_S$. We say that elements $h_1$ and $h_2$ are ``equivalent modulo $L$'', or that ``$h_1 \equiv h_2$ modulo $L$'', if $h_1-h_2 \in L$.

We call the three classes of generators given by $X_{i,j}^2$, $X_{i,j} X_{i,k}$ and $X_{i,j} X_{k,j}$ the \emph{monomial generators}, to distinguish them from $X_{i,j} X_{k,l} + X_{i,l} X_{k,j}$.  The first step is to show that it suffices to consider relations of plural multidegree whose terms only involve the monomial generators.

\begin{lemma} \label{thm:all-relations-are-monomial}
Consider an element of the form $f (g)  \in F_1$, where $f \in S$ and $g$ is a generator from \cref{thm:Shafiei}.  If $f(g) \in F_1$ is multi-homogeneous of plural multidegree, then it is equivalent modulo $L$ to an $S$-linear combination of monomial generators.
\end{lemma}

\begin{proof}
This statement is trivial if $g$ is a monomial generator, so assume that $g=X_{i,j} X_{k,l} + X_{i,l} X_{k,j}$.  Since the multidegree of $f(g)$ is plural, at least one of the rows or columns in the dots-and-boxes notation has either two dots or a dot and a corner of the rectangle. In other words, up to symmetry, the element $f(g) \in F_1$ must be divisible by one of
\begin{equation*}
\begin{relmatrix}{2}{2} \relgen{1}{1}{2}{2} \relcoeff{1}{1} \end{relmatrix}, \quad \begin{relmatrix}{2}{3} \relgen{1}{1}{2}{2} \relcoeff{1}{3} \end{relmatrix}, \quad \begin{relmatrix}{3}{3} \relgen{1}{1}{2}{2} \relcoeff[2]{3}{3} \end{relmatrix} \quad \text{or} \quad \begin{relmatrix}{3}{4} \relgen{1}{1}{2}{2} \relcoeff{3}{3} \relcoeff{3}{4} \end{relmatrix}.
\end{equation*}
It therefore suffices to show that each of these elements is equivalent to an $S$-linear combination of monomial generators modulo $L$.

The first two of these are straightforward: $\rho_4$ tells us that $\begin{relmatrix}{2}{2} \relgen{1}{1}{2}{2} \relcoeff{1}{1} \end{relmatrix} \equiv \begin{relmatrix}{2}{2} \relgen{1}{1}{1}{2} \relcoeff{2}{1} \end{relmatrix} + \begin{relmatrix}{2}{2} \relgen{1}{1}{1}{1} \relcoeff{2}{2} \end{relmatrix}$ modulo $L$, and $\rho_5$ says that $\begin{relmatrix}{2}{3} \relgen{1}{1}{2}{2} \relcoeff{1}{3} \end{relmatrix} \equiv \begin{relmatrix}{2}{3} \relgen{1}{1}{1}{3} \relcoeff{2}{2} \end{relmatrix} + \begin{relmatrix}{2}{3} \relgen{1}{2}{1}{3} \relcoeff{2}{1} \end{relmatrix}$ modulo $L$.

For the third case, observe that
\begin{align*}
\lefteqn{\begin{relmatrix}{3}{3} \relgen{1}{1}{2}{2} \relcoeff[2]{3}{3} \end{relmatrix} - \begin{relmatrix}{3}{3} \relgen{3}{3}{3}{3} \relcoeff{1}{1} \relcoeff{2}{2} \end{relmatrix} - \begin{relmatrix}{3}{3} \relgen{3}{3}{3}{3} \relcoeff{1}{2} \relcoeff{2}{1} \end{relmatrix}} \\
	& {} = \left( \begin{relmatrix}{3}{3} \relgen{2}{2}{3}{3} \relcoeff{1}{1} \relcoeff{3}{3} \end{relmatrix} - \begin{relmatrix}{3}{3} \relgen{3}{3}{3}{3} \relcoeff{1}{1} \relcoeff{2}{2} \end{relmatrix} - \begin{relmatrix}{3}{3} \relgen{3}{2}{3}{3} \relcoeff{1}{1} \relcoeff{2}{3} \end{relmatrix} \right) + \left( \begin{relmatrix}{3}{3} \relgen{2}{1}{3}{3} \relcoeff{1}{2} \relcoeff{3}{3} \end{relmatrix} - \begin{relmatrix}{3}{3} \relgen{3}{3}{3}{3} \relcoeff{1}{2} \relcoeff{2}{1} \end{relmatrix} - \begin{relmatrix}{3}{3} \relgen{3}{1}{3}{3} \relcoeff{1}{2} \relcoeff{2}{3} \end{relmatrix} \right) \\ 
	& \qquad {} - \left( \begin{relmatrix}{3}{3} \relgen{1}{1}{3}{2} \relcoeff{2}{3} \relcoeff{3}{3} \end{relmatrix} - \begin{relmatrix}{3}{3} \relgen{3}{2}{3}{3} \relcoeff{1}{1} \relcoeff{2}{3} \end{relmatrix} - \begin{relmatrix}{3}{3} \relgen{3}{1}{3}{3} \relcoeff{1}{2} \relcoeff{2}{3} \end{relmatrix} \right) - \left( \begin{relmatrix}{3}{3} \relgen{2}{2}{3}{3} \relcoeff{1}{1} \relcoeff{3}{3} \end{relmatrix} + \begin{relmatrix}{3}{3} \relgen{2}{1}{3}{3} \relcoeff{1}{2} \relcoeff{3}{3} \end{relmatrix} - \begin{relmatrix}{3}{3} \relgen{1}{1}{3}{2} \relcoeff{2}{3} \relcoeff{3}{3} \end{relmatrix} - \begin{relmatrix}{3}{3} \relgen{1}{1}{2}{2} \relcoeff[2]{3}{3} \end{relmatrix} \right)
\end{align*}
where each summand on the right hand side is in $L$: the first three are each permutations of $\rho_4$ times some $X_{i,j}$, and the fourth summand is $\rho_S$ times $X_{3,3}$. Therefore $\begin{relmatrix}{3}{3} \relgen{1}{1}{2}{2} \relcoeff[2]{3}{3} \end{relmatrix} \equiv \begin{relmatrix}{3}{3} \relgen{3}{3}{3}{3} \relcoeff{1}{1} \relcoeff{2}{2} \end{relmatrix} + \begin{relmatrix}{3}{3} \relgen{3}{3}{3}{3} \relcoeff{1}{2} \relcoeff{2}{1} \end{relmatrix}$ modulo $L$.

In the fourth case, we have
\begin{multline*}
\begin{relmatrix}{3}{4} \relgen{1}{1}{2}{2} \relcoeff{3}{3} \relcoeff{3}{4} \end{relmatrix} - \begin{relmatrix}{3}{4} \relgen{3}{3}{3}{4} \relcoeff{1}{1} \relcoeff{2}{2} \end{relmatrix} - \begin{relmatrix}{3}{4} \relgen{3}{3}{3}{4} \relcoeff{1}{2} \relcoeff{2}{1} \end{relmatrix} = \left( \begin{relmatrix}{3}{4} \relgen{2}{2}{3}{3} \relcoeff{1}{1} \relcoeff{3}{4} \end{relmatrix} - \begin{relmatrix}{3}{4} \relgen{3}{3}{3}{4} \relcoeff{1}{1} \relcoeff{2}{2} \end{relmatrix} - \begin{relmatrix}{3}{4} \relgen{3}{2}{3}{4} \relcoeff{1}{1} \relcoeff{2}{3} \end{relmatrix} \right) \\
	{} + \left( \begin{relmatrix}{3}{4} \relgen{2}{1}{3}{3} \relcoeff{1}{2} \relcoeff{3}{4} \end{relmatrix} - \begin{relmatrix}{3}{4} \relgen{3}{3}{3}{4} \relcoeff{1}{2} \relcoeff{2}{1} \end{relmatrix} - \begin{relmatrix}{3}{4} \relgen{3}{1}{3}{4} \relcoeff{1}{2} \relcoeff{2}{3} \end{relmatrix} \right) - \left( \begin{relmatrix}{3}{4} \relgen{1}{1}{3}{2} \relcoeff{2}{3} \relcoeff{3}{4} \end{relmatrix} - \begin{relmatrix}{3}{4} \relgen{3}{2}{3}{4} \relcoeff{1}{1} \relcoeff{2}{3} \end{relmatrix} - \begin{relmatrix}{3}{4} \relgen{3}{1}{3}{4} \relcoeff{1}{2} \relcoeff{2}{3} \end{relmatrix} \right) \\
	{} - \left( \begin{relmatrix}{3}{4} \relgen{2}{2}{3}{3} \relcoeff{1}{1} \relcoeff{3}{4} \end{relmatrix} + \begin{relmatrix}{3}{4} \relgen{2}{1}{3}{3} \relcoeff{1}{2} \relcoeff{3}{4} \end{relmatrix} - \begin{relmatrix}{3}{4} \relgen{1}{1}{3}{2} \relcoeff{2}{3} \relcoeff{3}{4} \end{relmatrix} - \begin{relmatrix}{3}{4} \relgen{1}{1}{2}{2} \relcoeff{3}{3} \relcoeff{3}{4} \end{relmatrix} \right) 
\end{multline*}
where each summand on the right hand side is in $L$ --- the first three are permutations of $\rho_5$ times some $X_{i,j}$, and the fourth is $\rho_S$ times $X_{3,4}$ --- so $\begin{relmatrix}{3}{4} \relgen{1}{1}{2}{2} \relcoeff{3}{3} \relcoeff{3}{4} \end{relmatrix} \equiv \begin{relmatrix}{3}{4} \relgen{3}{3}{3}{4} \relcoeff{1}{1} \relcoeff{2}{2} \end{relmatrix} + \begin{relmatrix}{3}{4} \relgen{3}{3}{3}{4} \relcoeff{1}{2} \relcoeff{2}{1} \end{relmatrix}$ modulo $L$.
\end{proof}

The first step in considering higher dimension relations is to examine the quadratic ones.

\begin{proposition} \label{thm:linear-relations-generate-quadratic}
If $\chartext(\bk) \neq 2$, then all quadratic relations of plural multidegree are generated by the linear relations.
\end{proposition}

\begin{proof}
Once again we work case by case; however, by \cref{thm:all-relations-are-monomial}, it suffices to consider relations only involving the monomial generators, which allows us to use the finer monomial grading instead of multigrading.  To establish this result, it suffices to show that all elements in $F_1$ of the form $f(g)$ with the same monomial degree, where $f$ is a monomial and $g$ is a monomial quadratic generator, are equivalent modulo $L$.

We will classify the monomial degrees by the multidegrees they are a refinement of.  We will use the matrix description to describe monomial degrees as introduced in \cref{S:gradings}.
\begin{itemize} [leftmargin=*]
 \setlength{\arraycolsep}{2.5pt} \renewcommand{\arraystretch}{0.7}
\item $(4, \ 4)$. There is only one monomial degree possible with this multidegree (up to symmetry), namely $\begin{bmatrix} 4 \end{bmatrix}$ (displayed in matrix form), and the only possible term with this degree is $X_{1,1}^2 (X_{1,1}^2)$, i.e.,  $\begin{relmatrix}{1}{1} \relgen{1}{1}{1}{1} \relcoeff[2]{1}{1} \end{relmatrix}$ in the dots-and-boxes notation.

\item $(4, \ 3+1)$. The only monomial degree of this multidegree is $\begin{bmatrix} 3 & 1 \end{bmatrix}$. The two possible terms with this degree are $\begin{relmatrix}{1}{2} \relgen{1}{1}{1}{1} \relcoeff{1}{1} \relcoeff{1}{2} \end{relmatrix}$ and $\begin{relmatrix}{1}{2} \relgen{1}{1}{1}{2} \relcoeff[2]{1}{1} \end{relmatrix}$, which are equivalent by $\rho_1$.

\item $(4, \ 2+2)$. The only monomial degree is $\begin{bmatrix} 2 & 2 \end{bmatrix}$. We have
\begin{align*}
\begin{relmatrix}{1}{2} \relgen{1}{1}{1}{1} \relcoeff[2]{1}{2} \end{relmatrix} & \equiv \begin{relmatrix}{1}{2} \relgen{1}{1}{1}{2} \relcoeff{1}{1} \relcoeff{1}{2} \end{relmatrix} && \text{(by $\rho_1$)} \\
& \equiv \begin{relmatrix}{1}{2} \relgen{1}{2}{1}{2} \relcoeff[2]{1}{1} \end{relmatrix} && \text{(by $\rho_1$)}
\end{align*}
and these are all the terms with this monomial degree.

\item $(4, \ 2+1+1)$. The only monomial degree is $\begin{bmatrix} 2 & 1 & 1 \end{bmatrix}$. We have
\begin{align*}
\begin{relmatrix}{1}{3} \relgen{1}{1}{1}{1} \relcoeff{1}{2} \relcoeff{1}{3} \end{relmatrix} & \equiv \begin{relmatrix}{1}{3} \relgen{1}{1}{1}{2} \relcoeff{1}{1} \relcoeff{1}{3} \end{relmatrix} && \text{(by $\rho_1$)} \\
& \equiv \begin{relmatrix}{1}{3} \relgen{1}{1}{1}{3} \relcoeff{1}{1} \relcoeff{1}{2} \end{relmatrix} && \text{(by $\rho_2$)} \\
& \equiv \begin{relmatrix}{1}{3} \relgen{1}{2}{1}{3} \relcoeff[2]{1}{1} \end{relmatrix} && \text{(by $\rho_2$)}
\end{align*}
and these are all the terms with this monomial degree.

\item $(4, \ 1+1+1+1)$. The only monomial degree is $\begin{bmatrix} 1 & 1 & 1 & 1 \end{bmatrix}$. The terms with this monomial degree are the permutations of $\begin{relmatrix}{1}{4} \relgen{1}{1}{1}{2} \relcoeff{1}{3} \relcoeff{1}{4} \end{relmatrix}$. Any two of these permutations that have a $\begin{relbare} \relcoeff{1}{1} \end{relbare}$ in the same position are equivalent by $\rho_2$, and any of the permutations that don't share a $\begin{relbare} \relcoeff{1}{1} \end{relbare}$ are each mutually equivalent to another of the permutations by the same rule: e.g.
\begin{equation*}
\begin{relmatrix}{1}{4} \relgen{1}{1}{1}{2} \relcoeff{1}{3} \relcoeff{1}{4} \end{relmatrix} \equiv \begin{relmatrix}{1}{4} \relgen{1}{2}{1}{3} \relcoeff{1}{1} \relcoeff{1}{4} \end{relmatrix} \equiv \begin{relmatrix}{1}{4} \relgen{1}{3}{1}{4} \relcoeff{1}{1} \relcoeff{1}{2} \end{relmatrix}
\end{equation*}

\item $(3+1, \ 3+1)$. The two monomial degrees for this multidegree are
\begin{itemize}
\item $\begin{bmatrix} 3 & \\ & 1 \end{bmatrix}$: in this case, the only monomial term is $\begin{relmatrix}{2}{2} \relgen{1}{1}{1}{1} \relcoeff{1}{1} \relcoeff{2}{2} \end{relmatrix}$.
\item $\begin{bmatrix} 2 & 1 \\ 1 & \end{bmatrix}$: we have
\begin{align*}
\begin{relmatrix}{2}{2} \relgen{1}{1}{1}{2} \relcoeff{1}{1} \relcoeff{2}{1} \end{relmatrix} & \equiv \begin{relmatrix}{2}{2} \relgen{1}{1}{1}{1} \relcoeff{1}{2} \relcoeff{2}{1} \end{relmatrix} \equiv \begin{relmatrix}{2}{2} \relgen{1}{1}{2}{1} \relcoeff{1}{1} \relcoeff{1}{2} \end{relmatrix} && \text{(by $\rho_1$)}
\end{align*}
\end{itemize}

\item $(3+1, \ 2+2)$. The only monomial degree up to symmetry here is $\begin{bmatrix} 2 & 1 \\ & 1 \end{bmatrix}$. We have
\begin{align*}
\begin{relmatrix}{2}{2} \relgen{1}{1}{1}{1} \relcoeff{1}{2} \relcoeff{2}{2} \end{relmatrix} & \equiv \begin{relmatrix}{2}{2} \relgen{1}{1}{1}{2} \relcoeff{1}{1} \relcoeff{2}{2} \end{relmatrix} && \text{(by $\rho_1$)} \\
& \equiv \begin{relmatrix}{2}{2} \relgen{1}{2}{2}{2} \relcoeff[2]{1}{1} \end{relmatrix} && \text{(by $\rho_3$)}
\end{align*}

\item $(3+1, \ 2+1+1)$. The monomial degrees up to symmetry are:
\begin{itemize}
\item $\begin{bmatrix} 2 & 1 & \\ & & 1 \end{bmatrix}$: here we have
\begin{align*}
\begin{relmatrix}{2}{3} \relgen{1}{1}{1}{1} \relcoeff{1}{2} \relcoeff{2}{3} \end{relmatrix} & \equiv \begin{relmatrix}{2}{3} \relgen{1}{1}{1}{2} \relcoeff{1}{1} \relcoeff{2}{3} \end{relmatrix} && \text{(by $\rho_1$)}
\end{align*}
\item $\begin{bmatrix} 1 & 1 & 1 \\ 1 & & \end{bmatrix}$: we have
\begin{align*}
\begin{relmatrix}{2}{3} \relgen{1}{1}{2}{1} \relcoeff{1}{2} \relcoeff{1}{3} \end{relmatrix} & \equiv \begin{relmatrix}{2}{3} \relgen{1}{1}{1}{2} \relcoeff{2}{1} \relcoeff{1}{3} \end{relmatrix} && \text{(by $\rho_3$)} \\
& \equiv \begin{relmatrix}{2}{3} \relgen{1}{1}{1}{3} \relcoeff{2}{1} \relcoeff{1}{2} \end{relmatrix} \equiv \begin{relmatrix}{2}{3} \relgen{1}{2}{1}{3} \relcoeff{2}{1} \relcoeff{1}{1} \end{relmatrix} && \text{(by $\rho_2$)}
\end{align*}
\end{itemize}

\item $(3+1, \ 1+1+1+1)$. Up to symmetry the only monomial degree is $\begin{bmatrix} 1 & 1 & 1 & \\ & & & 1 \end{bmatrix}$. But
\begin{align*}
\begin{relmatrix}{2}{4} \relgen{1}{1}{1}{2} \relcoeff{1}{3} \relcoeff{2}{4} \end{relmatrix} & \equiv \begin{relmatrix}{2}{4} \relgen{1}{1}{1}{3} \relcoeff{1}{2} \relcoeff{2}{4} \end{relmatrix} \equiv \begin{relmatrix}{2}{4} \relgen{1}{2}{1}{3} \relcoeff{1}{1} \relcoeff{2}{4} \end{relmatrix} && \text{(by $\rho_2$)}
\end{align*}

\item $(2+2, \ 2+2)$. The monomial degrees are
\begin{itemize}
\item $\begin{bmatrix} 2 & \\ & 2 \end{bmatrix}$: by adding relations from $L$, specifically permutations of $\rho_4$ and $\rho_3$, we can obtain
\begin{multline*}
\left( \begin{relmatrix}{2}{2} \relgen{1}{1}{2}{2} \relcoeff{1}{1} \relcoeff{2}{2} \end{relmatrix} - \begin{relmatrix}{2}{2} \relgen{1}{2}{2}{2} \relcoeff{1}{1} \relcoeff{2}{1} \end{relmatrix} - \begin{relmatrix}{2}{2} \relgen{2}{2}{2}{2} \relcoeff[2]{1}{1} \end{relmatrix} \right) - \left( \begin{relmatrix}{2}{2} \relgen{1}{1}{2}{2} \relcoeff{1}{1} \relcoeff{2}{2} \end{relmatrix} - \begin{relmatrix}{2}{2} \relgen{1}{1}{1}{2} \relcoeff{2}{1} \relcoeff{2}{2} \end{relmatrix} - \begin{relmatrix}{2}{2} \relgen{1}{1}{1}{1} \relcoeff[2]{2}{2} \end{relmatrix} \right) + \left( \begin{relmatrix}{2}{2} \relgen{1}{2}{2}{2} \relcoeff{1}{1} \relcoeff{2}{1} \end{relmatrix} - \begin{relmatrix}{2}{2} \relgen{1}{1}{1}{2} \relcoeff{2}{1} \relcoeff{2}{2} \end{relmatrix} \right) \\ {} = \begin{relmatrix}{2}{2} \relgen{1}{1}{1}{1} \relcoeff[2]{2}{2} \end{relmatrix} - \begin{relmatrix}{2}{2} \relgen{2}{2}{2}{2} \relcoeff[2]{1}{1} \end{relmatrix}
\end{multline*}
and thus $\begin{relmatrix}{2}{2} \relgen{1}{1}{1}{1} \relcoeff[2]{2}{2} \end{relmatrix} \equiv \begin{relmatrix}{2}{2} \relgen{2}{2}{2}{2} \relcoeff[2]{1}{1} \end{relmatrix}$ modulo $L$. These are the only terms possible with this monomial degree.
\item $\begin{bmatrix} 1 & 1 \\ 1 & 1 \end{bmatrix}$: we have
\begin{align*}
\begin{relmatrix}{2}{2} \relgen{2}{1}{2}{2} \relcoeff{1}{1} \relcoeff{1}{2} \end{relmatrix} & \equiv \begin{relmatrix}{2}{2} \relgen{1}{2}{2}{2} \relcoeff{1}{1} \relcoeff{2}{1} \end{relmatrix} \equiv \begin{relmatrix}{2}{2} \relgen{1}{1}{1}{2} \relcoeff{2}{1} \relcoeff{2}{2} \end{relmatrix} \equiv \begin{relmatrix}{2}{2} \relgen{1}{1}{2}{1} \relcoeff{1}{2} \relcoeff{2}{2} \end{relmatrix} && \text{(by $\rho_3$)}
\end{align*}
\end{itemize}

\item $(2+2, \ 2+1+1)$. The possible monomial degrees are:
\begin{itemize}
\item $\begin{bmatrix} 2 & & \\ & 1 & 1 \end{bmatrix}$: we can use $\rho_5$, $\rho_4$ and $\rho_3$ to write:
\begin{multline*}
\left( \begin{relmatrix}{2}{3} \relgen{1}{1}{2}{2} \relcoeff{1}{1} \relcoeff{2}{3} \end{relmatrix} - \begin{relmatrix}{2}{3} \relgen{2}{1}{2}{3} \relcoeff{1}{1} \relcoeff{1}{2} \end{relmatrix} - \begin{relmatrix}{2}{3} \relgen{2}{2}{2}{3} \relcoeff[2]{1}{1} \end{relmatrix} \right) - \left( \begin{relmatrix}{2}{3} \relgen{1}{1}{2}{2} \relcoeff{1}{1} \relcoeff{2}{3} \end{relmatrix} - \begin{relmatrix}{2}{3} \relgen{1}{1}{1}{1} \relcoeff{2}{2} \relcoeff{2}{3} \end{relmatrix} - \begin{relmatrix}{2}{3} \relgen{1}{1}{2}{1} \relcoeff{1}{2} \relcoeff{2}{3} \end{relmatrix} \right) \\ {} + \left( \begin{relmatrix}{2}{3} \relgen{2}{1}{2}{3} \relcoeff{1}{1} \relcoeff{1}{2} \end{relmatrix} - \begin{relmatrix}{2}{3} \relgen{1}{1}{2}{1} \relcoeff{1}{2} \relcoeff{2}{3} \end{relmatrix} \right) = \begin{relmatrix}{2}{3} \relgen{1}{1}{1}{1} \relcoeff{2}{2} \relcoeff{2}{3} \end{relmatrix} - \begin{relmatrix}{2}{3} \relgen{2}{2}{2}{3} \relcoeff[2]{1}{1} \end{relmatrix}
\end{multline*}
so $\begin{relmatrix}{2}{3} \relgen{1}{1}{1}{1} \relcoeff{2}{2} \relcoeff{2}{3} \end{relmatrix} \equiv \begin{relmatrix}{2}{3} \relgen{2}{2}{2}{3} \relcoeff[2]{1}{1} \end{relmatrix}$ modulo $L$, and these are the only terms with this monomial degree.
\item $\begin{bmatrix} 1 & 1 & \\ 1 & & 1 \end{bmatrix}$: we have
\begin{align*}
\begin{relmatrix}{2}{3} \relgen{1}{1}{1}{2} \relcoeff{2}{1} \relcoeff{2}{3} \end{relmatrix} & \equiv \begin{relmatrix}{2}{3} \relgen{1}{1}{2}{1} \relcoeff{1}{2} \relcoeff{2}{3} \end{relmatrix} \equiv \begin{relmatrix}{2}{3} \relgen{2}{1}{2}{3} \relcoeff{1}{1} \relcoeff{1}{2} \end{relmatrix} && \text{(by $\rho_3$)}
\end{align*}
\end{itemize}

\item $(2+2, \ 1+1+1+1)$. The only monomial degree up to symmetry is $\begin{bmatrix} 1 & 1 & & \\ & & 1 & 1 \end{bmatrix}$. We can form the following sum using $\rho_5$:
\begin{multline*}
\left(
\begin{relmatrix}{2}{4} \relcoeff{1}{1} \relcoeff{2}{4} \relgen{1}{2}{2}{3} \end{relmatrix} - \begin{relmatrix}{2}{4} \relcoeff{1}{1} \relcoeff{1}{2} \relgen{2}{3}{2}{4} \end{relmatrix} - \begin{relmatrix}{2}{4} \relcoeff{1}{1} \relcoeff{1}{3} \relgen{2}{2}{2}{4} \end{relmatrix} \right) - \left( \begin{relmatrix}{2}{4} \relcoeff{1}{1} \relcoeff{2}{4} \relgen{1}{2}{2}{3} \end{relmatrix} - \begin{relmatrix}{2}{4} \relcoeff{2}{3} \relcoeff{2}{4} \relgen{1}{1}{1}{2} \end{relmatrix} - \begin{relmatrix}{2}{4} \relcoeff{2}{2} \relcoeff{2}{4} \relgen{1}{1}{1}{3} \end{relmatrix} \right) \\ {} + \left( \begin{relmatrix}{2}{4} \relcoeff{1}{3} \relcoeff{2}{4} \relgen{1}{1}{2}{2} \end{relmatrix} - \begin{relmatrix}{2}{4} \relcoeff{2}{1} \relcoeff{2}{4} \relgen{1}{2}{1}{3} \end{relmatrix} - \begin{relmatrix}{2}{4} \relcoeff{2}{2} \relcoeff{2}{4} \relgen{1}{1}{1}{3} \end{relmatrix} \right) - \left( \begin{relmatrix}{2}{4} \relcoeff{1}{3} \relcoeff{2}{4} \relgen{1}{1}{2}{2} \end{relmatrix} - \begin{relmatrix}{2}{4} \relcoeff{1}{2} \relcoeff{1}{3} \relgen{2}{1}{2}{4} \end{relmatrix} - \begin{relmatrix}{2}{4} \relcoeff{1}{1} \relcoeff{1}{3} \relgen{2}{2}{2}{4} \end{relmatrix} \right) \\ {} + \left( \begin{relmatrix}{2}{4} \relcoeff{1}{2} \relcoeff{2}{4} \relgen{1}{1}{2}{3} \end{relmatrix} -  \begin{relmatrix}{2}{4} \relcoeff{1}{2} \relcoeff{1}{3} \relgen{2}{1}{2}{4} \end{relmatrix} - \begin{relmatrix}{2}{4} \relcoeff{1}{1} \relcoeff{1}{2} \relgen{2}{3}{2}{4} \end{relmatrix} \right) - \left( \begin{relmatrix}{2}{4} \relcoeff{1}{2} \relcoeff{2}{4} \relgen{1}{1}{2}{3} \end{relmatrix} - \begin{relmatrix}{2}{4} \relcoeff{2}{1} \relcoeff{2}{4} \relgen{1}{2}{1}{3} \end{relmatrix} - \begin{relmatrix}{2}{4} \relcoeff{2}{3} \relcoeff{2}{4} \relgen{1}{1}{1}{2} \end{relmatrix} \right) \\ {} = 2 \left( \begin{relmatrix}{2}{4}
\relcoeff{2}{3} \relcoeff{2}{4} \relgen{1}{1}{1}{2} \end{relmatrix} - \begin{relmatrix}{2}{4} \relcoeff{1}{1} \relcoeff{1}{2} \relgen{2}{3}{2}{4} \end{relmatrix} \right)
\end{multline*}
and since  $\chartext{\bk} \neq 2$, we must have $\begin{relmatrix}{2}{4} \relcoeff{2}{3} \relcoeff{2}{4} \relgen{1}{1}{1}{2} \end{relmatrix} \equiv \begin{relmatrix}{2}{4} \relcoeff{1}{1} \relcoeff{1}{2} \relgen{2}{3}{2}{4} \end{relmatrix}$ modulo $L$. (This is the only place in the argument where the hypothesis that $\chartext{\bk} \neq 2$ is used.)

\item $(2+1+1, \ 2+1+1)$. The possible monomial degrees here are
\begin{itemize}
\item $\begin{bmatrix} 2 \\ & 1 \\ & & 1 \end{bmatrix}$: in this case, the only possible term is $\begin{relmatrix}{3}{3} \relgen{1}{1}{1}{1} \relcoeff{2}{2} \relcoeff{3}{3} \end{relmatrix}$.
\item $\begin{bmatrix} 1 & 1 \\ 1 \\ & & 1 \end{bmatrix}$: we have
\begin{align*}
\begin{relmatrix}{3}{3} \relgen{1}{1}{1}{2} \relcoeff{2}{1} \relcoeff{3}{3} \end{relmatrix} & \equiv \begin{relmatrix}{3}{3} \relgen{1}{1}{2}{1} \relcoeff{1}{2} \relcoeff{3}{3} \end{relmatrix} && \text{(by $\rho_3$)}
\end{align*}
\item $\begin{bmatrix} & 1 & 1 \\ 1 \\ 1 \end{bmatrix}$: using $\rho_5$ and $\rho_3$ we can write
\begin{multline*}
\left( \begin{relmatrix}{3}{3} \relgen{1}{1}{2}{2} \relcoeff{1}{3} \relcoeff{3}{1} \end{relmatrix} - \begin{relmatrix}{3}{3} \relgen{1}{1}{3}{1} \relcoeff{1}{3} \relcoeff{2}{2} \end{relmatrix} - \begin{relmatrix}{3}{3} \relgen{2}{1}{3}{1} \relcoeff{1}{2} \relcoeff{1}{3} \end{relmatrix} \right) - \left( \begin{relmatrix}{3}{3} \relgen{1}{1}{2}{2} \relcoeff{1}{3} \relcoeff{3}{1} \end{relmatrix} - \begin{relmatrix}{3}{3} \relgen{1}{1}{1}{3} \relcoeff{2}{2} \relcoeff{3}{1} \end{relmatrix} - \begin{relmatrix}{3}{3} \relgen{1}{2}{1}{3} \relcoeff{2}{1} \relcoeff{3}{1} \end{relmatrix} \right) \\ {} + \left( \begin{relmatrix}{3}{3} \relgen{1}{1}{3}{1} \relcoeff{1}{3} \relcoeff{2}{2} \end{relmatrix} - \begin{relmatrix}{3}{3} \relgen{1}{1}{1}{3} \relcoeff{2}{2} \relcoeff{3}{1} \end{relmatrix} \right) = \begin{relmatrix}{3}{3} \relgen{1}{2}{1}{3} \relcoeff{2}{1} \relcoeff{3}{1} \end{relmatrix} - \begin{relmatrix}{3}{3} \relgen{2}{1}{3}{1} \relcoeff{1}{2} \relcoeff{1}{3} \end{relmatrix}
\end{multline*}
so $\begin{relmatrix}{3}{3} \relgen{1}{2}{1}{3} \relcoeff{2}{1} \relcoeff{3}{1} \end{relmatrix} \equiv \begin{relmatrix}{3}{3} \relgen{2}{1}{3}{1} \relcoeff{1}{2} \relcoeff{1}{3} \end{relmatrix}$ modulo $L$.
\end{itemize}

\item $(2+1+1, \ 1+1+1+1)$. Up to symmetry, the only possible monomial degree is $\begin{bmatrix} 1 & 1 \\ & & 1 \\ & & & 1 \end{bmatrix}$, and $\begin{relmatrix}{3}{4} \relgen{1}{1}{1}{2} \relcoeff{2}{3} \relcoeff{3}{4} \end{relmatrix}$ is the only term with this degree.
\end{itemize}
This covers all cases, up to symmetry.
\end{proof}

We now generalize this to all plural relations.

\begin{proposition} \label{thm:linear-generate-all-plural}
If $\chartext(\bk) \neq 2$, then all relations of plural multidegree are generated by the linear relations.
\end{proposition}

\begin{proof}
By using \cref{thm:all-relations-are-monomial}, it suffices to consider relations involving monomial generators.  As in the proof of \cref{thm:linear-relations-generate-quadratic}, it suffices to show that any two elements $h_1=f_1(g_1)$ and $h_2=f_2(g_2)$ of the same monomial degree, where each $f_i$ is a monomial and each $g_i$ is a monomial quadratic generator, are equivalent modulo $L$. Since $h_1$ and $h_2$ have the same monomial degree, they involve the same variables, with multiplicity.  By \cref{thm:linear-generate-all-singular} and \cref{thm:linear-relations-generate-quadratic}, we may assume that the standard degree $m$ of $h_1$ and $h_2$ is greater than $4$. Note that exactly two variables appear in each generator, so $f_1$ and $f_2$ each involve $m-2$ variables. Therefore, by the pigeonhole principle, at least $m-4$ of the variables must appear in both $f_1$ and $f_2$. Thus we can write
\begin{equation*}
h_1 - h_2 = p \cdot (h'_1 - h'_2)
\end{equation*}
where $p$ is a degree $m-4$ monomial in $S$, and $h'_1$ and $h'_2$ are elements in $F_1$ which have plural multidegree and standard degree $4$, and only involve monomial generators. But \cref{thm:linear-relations-generate-quadratic} says that $h'_1$ and $h'_2$ are equivalent modulo $L$, and therefore $h_1$ and $h_2$ are equivalent modulo $L$. 
\end{proof}

This finally establishes our main result:

\begin{proof}[Proof of \cref{main-theorem}\ref{main-theorem-a}]
\Cref{thm:linear-generate-all-singular} covers the singular case, and \cref{thm:linear-generate-all-plural} covers the plural case.  The dimension of the space of linear relations is computed in \cref{thm:beta23}.
\end{proof}

\begin{corollary}
If $\chartext(\bk) \neq 2$, then all Betti numbers $\beta_{2,j}$ for $j \neq 3$ are zero. 
\end{corollary}

\section{Analogous results for the permanent} \label{sec:permanent}

Most of the results from \cref{sec:singular-multidegree,sec:plural-multidegree} also apply to the apolar ideal of $\perm_n$ with some slight modifications.
Recall from \cref{thm:Shafiei} that the generators of $\perm_n^\perp$ include polynomials of the form $X_{i,j} X_{k,l} - X_{i,l} X_{k,j}$ (instead of $X_{i,j} X_{k,l} + X_{i,l} X_{k,j}$) as well as the monomial generators $X_{i,j}^2$, $X_{i,j} X_{i,k}$ and $X_{i,j} X_{k,j}$.  The main result of this section is \cref{thm:linear+Q-generate-all-permanent} characterizing the module of relations between these generators.  

\subsection{Keeping track of parity}
In \cref{sec:singular-multidegree}, we associated singular multidegree terms with generator $X_{i,j} X_{k,l} + X_{i,l} X_{k,j}$ to edges of the Cayley graph $\Gamma(\bS_m)$, but with the generator $X_{i,j} X_{k,l} - X_{i,l} X_{k,j}$ instead, we must take care to distinguish between $X_{i,j} X_{k,l} - X_{i,l} X_{k,j}$ and its negative, $X_{i,l} X_{k,j} - X_{i,j} X_{k,l}$. To ensure consistency, given a term $f \cdot (X_{i,j} X_{k,l} - X_{i,l} X_{k,j})$ where $f$ is a monomial, we choose its sign so that the monomials $f X_{i,j} X_{k,l}$ and $f X_{i,l} X_{k,j}$ get the same sign as the parity of the corresponding permutations of $\bS_m$.

With this convention, the coefficient of a monomial corresponding to an odd permutation is the \emph{negative} of the sum of the weights of edges meeting the corresponding vertex of the Cayley graph. But this still gives rise to a relation if and only if these weights sum to zero at every vertex, so relations still correspond to zero-magic graphs.  Thus \cref{lem:relations} still holds and the rest of \cref{sec:singular-multidegree} applies intact.  
In particular, the relation corresponding to the prototypical commutator cycle $(1 \ 2) (1 \ 3) (1 \ 2) (2 \ 3)$ is
\begin{multline*}
\rho_S' = X_{3,3} (X_{1,1} X_{2,2} - X_{1,2} X_{2,1}) - X_{1,2} (X_{2,3} X_{3,1} - X_{2,1} X_{3,3}) \\ {} + X_{2,3} (X_{1,2} X_{3,1} - X_{1,1} X_{3,2}) - X_{1,1} (X_{2,2} X_{3,3} - X_{2,3} X_{3,2}).
\end{multline*}

This issue of parity applies to the dots-and-rectangles notation as well. To remove ambiguity, when discussing the permanent, the generator $X_{i,j} X_{k,l} - X_{i,l} X_{k,j}$ will be shown as a rectangle with a dotted line connecting the corners from the monomial with positive sign. Thus
\begin{align*}
(X_{1,1} X_{2,2} - X_{1,2} X_{2,1}) & = \begin{relmatrix}{2}{2} \relgenperm{1}{1}{2}{2} \end{relmatrix} \\
(X_{1,2} X_{2,1} - X_{1,1} X_{2,2}) & = \begin{relmatrix}{2}{2} \relgenpermT{1}{1}{2}{2} \end{relmatrix} = - \begin{relmatrix}{2}{2} \relgenperm{1}{1}{2}{2} \end{relmatrix}
\end{align*}

\subsection{Linear relations}
In \cref{sec:plural-multidegree}, we must modify the six linear relations from \cref{thm:linear-relations-list} to the linear relations:
\begin{equation} \label{E:linear-perm}
\begin{aligned}
\rho_1' &= X_{1,2} (X_{1,1}^2) - X_{1,1} (X_{1,1} X_{1,2}) \\
\rho_2' &= X_{1,3} (X_{1,1} X_{1,2}) - X_{1,2} (X_{1,1} X_{1,3}) \\
\rho_3' &= X_{2,1} (X_{1,1} X_{1,2}) - X_{1,2} (X_{1,1} X_{2,1}) \\
\rho_4' &= X_{1,1} (X_{1,1} X_{2,2} - X_{1,2} X_{2,1}) + X_{2,1} (X_{1,1} X_{1,2}) - X_{2,2} (X_{1,1}^2) \\
\rho_5' &= X_{1,3} (X_{1,1} X_{2,2} - X_{1,2} X_{2,1}) - X_{2,2} (X_{1,1} X_{1,3}) + X_{2,1} (X_{1,2} X_{1,3}) \\
\rho_S' &= \begin{multlined}[t] X_{3,3} (X_{1,1} X_{2,2} - X_{1,2} X_{2,1}) - X_{1,2} (X_{2,3} X_{3,1} - X_{2,1} X_{3,3}) \\[2.5pt] {} + X_{2,3} (X_{1,2} X_{3,1} - X_{1,1} X_{3,2}) - X_{1,1} (X_{2,2} X_{3,3} - X_{2,3} X_{3,2}).
\end{multlined}
\end{aligned}
\end{equation}
\cref{thm:linear-relations-list,thm:beta23,thm:all-relations-are-monomial} hold with only minor modifications.

However, \cref{thm:linear-relations-generate-quadratic} is not true for the permanent. Recall also that \cref{thm:linear-relations-generate-quadratic} breaks down if $\chartext(\bk)=2$, and the place where the proof breaks is showing that $\begin{relmatrix}{2}{4} \relcoeff{2}{3} \relcoeff{2}{4} \relgen{1}{1}{1}{2} \end{relmatrix} \equiv \begin{relmatrix}{2}{4} \relcoeff{1}{1} \relcoeff{1}{2} \relgen{2}{3}{2}{4} \end{relmatrix}$ modulo $L$ in the case of multidegree $(2+2, \ 1+1+1+1)$. It turns out that the for the apolar ideal $\perm_n^{\perp}$ in arbitrary characteristic  (as well as $\det_n^{\perp}$ if $\chartext(\bk)=2$), there are quadratic relations not generated by linear relations.

\begin{proposition} \label{thm:quadratic-generator-for-perm}
For $\perm_n^\perp$ in arbitrary characteristic, the quadratic relation
\begin{equation} \label{E:quadratic-perm}
\rho_Q' = X_{2,3} X_{2,4} (X_{1,1} X_{1,2}) - X_{1,1} X_{1,2} (X_{2,3} X_{2,4}) = \begin{relmatrix}{2}{4} \relcoeff{2}{3} \relcoeff{2}{4} \relgen{1}{1}{1}{2} \end{relmatrix} - \begin{relmatrix}{2}{4} \relcoeff{1}{1} \relcoeff{1}{2} \relgen{2}{3}{2}{4} \end{relmatrix}
\end{equation}
is \emph{not} generated by the linear relations \eqref{E:linear-perm}.
\end{proposition}

\begin{proof}
Let $L'$ be the submodule generated by the orbit of $\rho_1', \dotsc, \rho_5', \rho_S'$ (analogous to $L$, defined earlier).  The multidegree of the relation $\rho_Q'$ is $(2+2, \ 1+1+1+1)$. Therefore, if this relation can be written as a homogeneous sum of the linear relations, the only linear relation that can be involved is $\rho_5'$: the relations $\rho_1'$ and $\rho_2'$ have multidegree $3$ in one row, $\rho_3'$ and $\rho_4'$ have multidegree $2$ in both a row and a column, and $\rho_S'$ has positive multidegree in three rows and three columns.

Under the action of $\bS_2 \times \bS_4$ on the $2 \times 4$ submatrix where this relation lives, the possible terms with this multidegree form two distinct orbits:
\begin{equation*}
\begin{relmatrix}{2}{4} \relgen{1}{1}{1}{2} \relcoeff{2}{3} \relcoeff{2}{4} \end{relmatrix} \quad \text{and} \quad \begin{relmatrix}{2}{4} \relgenperm{1}{2}{2}{3} \relcoeff{1}{1} \relcoeff{2}{4} \end{relmatrix}
\end{equation*}
are representative elements of each.

The orbit of the linear relation $\rho_5'$ can generate only two relations (up to scaling) that contain the term $\begin{relmatrix}{2}{4} \relgenperm{1}{2}{2}{3} \relcoeff{1}{1} \relcoeff{2}{4} \end{relmatrix}$, made by taking each of the dots as constant: they are
\begin{align*}
\begin{relmatrix}{2}{4} \relgenperm{1}{2}{2}{3} \relcoeff{1}{1} \relcoeff{2}{4} \end{relmatrix} - \begin{relmatrix}{2}{4} \relgen{1}{1}{1}{2} \relcoeff{2}{3} \relcoeff{2}{4} \end{relmatrix} + \begin{relmatrix}{2}{4} \relgen{1}{1}{1}{3} \relcoeff{2}{2} \relcoeff{2}{4} \end{relmatrix} & = X_{2,4} \left( \begin{relmatrix}{2}{4} \relgenperm{1}{2}{2}{3} \relcoeff{1}{1} \end{relmatrix} - \begin{relmatrix}{2}{4} \relgen{1}{1}{1}{2} \relcoeff{2}{3} \end{relmatrix} + \begin{relmatrix}{2}{4} \relgen{1}{1}{1}{3} \relcoeff{2}{2} \end{relmatrix} \right) \\
\shortintertext{and}
\begin{relmatrix}{2}{4} \relgenperm{1}{2}{2}{3} \relcoeff{1}{1} \relcoeff{2}{4} \end{relmatrix} - \begin{relmatrix}{2}{4} \relgen{2}{3}{2}{4} \relcoeff{1}{1} \relcoeff{1}{2} \end{relmatrix} + \begin{relmatrix}{2}{4} \relgen{2}{2}{2}{4} \relcoeff{1}{1} \relcoeff{1}{3} \end{relmatrix} & = X_{1,1} \left( \begin{relmatrix}{2}{4} \relgenperm{1}{2}{2}{3} \relcoeff{2}{4} \end{relmatrix} - \begin{relmatrix}{2}{4} \relgen{2}{3}{2}{4} \relcoeff{1}{2} \end{relmatrix} + \begin{relmatrix}{2}{4} \relgen{2}{2}{2}{4} \relcoeff{1}{3} \end{relmatrix} \right).
\end{align*}
Suppose some $\bk$-linear combination of permutations of these relations gives $\rho_Q'$. Since $\begin{relmatrix}{2}{4} \relgenperm{1}{2}{2}{3} \relcoeff{1}{1} \relcoeff{2}{4} \end{relmatrix}$ does not appear in $\rho_Q'$, if one of these has scalar coefficient $c$, the other must have coefficient $-c$.

Therefore, if $\rho_Q'$ is generated by the linear relations, it must be a $\bk$-linear combination of the $\mathbf S_2 \times \mathbf S_4$ orbit of
\begin{multline*}
\left( \begin{relmatrix}{2}{4} \relgenperm{1}{2}{2}{3} \relcoeff{1}{1} \relcoeff{2}{4} \end{relmatrix} - \begin{relmatrix}{2}{4} \relgen{2}{3}{2}{4} \relcoeff{1}{1} \relcoeff{1}{2} \end{relmatrix} + \begin{relmatrix}{2}{4} \relgen{2}{2}{2}{4} \relcoeff{1}{1} \relcoeff{1}{3} \end{relmatrix} \right) - \left( \begin{relmatrix}{2}{4} \relgenperm{1}{2}{2}{3} \relcoeff{1}{1} \relcoeff{2}{4} \end{relmatrix} - \begin{relmatrix}{2}{4} \relgen{1}{1}{1}{2} \relcoeff{2}{3} \relcoeff{2}{4} \end{relmatrix} + \begin{relmatrix}{2}{4} \relgen{1}{1}{1}{3} \relcoeff{2}{2} \relcoeff{2}{4} \end{relmatrix} \right) \qquad \qquad \\
= \begin{relmatrix}{2}{4} \relgen{1}{1}{1}{2} \relcoeff{2}{3} \relcoeff{2}{4} \end{relmatrix} - \begin{relmatrix}{2}{4} \relgen{2}{3}{2}{4} \relcoeff{1}{1} \relcoeff{1}{2} \end{relmatrix} - \begin{relmatrix}{2}{4} \relgen{1}{1}{1}{3} \relcoeff{2}{2} \relcoeff{2}{4} \end{relmatrix} + \begin{relmatrix}{2}{4} \relgen{2}{2}{2}{4} \relcoeff{1}{1} \relcoeff{1}{3} \end{relmatrix}.
\end{multline*}
This relation tells us that $\begin{relmatrix}{2}{4} \relgen{1}{1}{1}{2} \relcoeff{2}{3} \relcoeff{2}{4} \end{relmatrix} - \begin{relmatrix}{2}{4} \relgen{2}{3}{2}{4} \relcoeff{1}{1} \relcoeff{1}{2} \end{relmatrix} \equiv \begin{relmatrix}{2}{4} \relgen{1}{1}{1}{3} \relcoeff{2}{2} \relcoeff{2}{4} \end{relmatrix} - \begin{relmatrix}{2}{4} \relgen{2}{2}{2}{4} \relcoeff{1}{1} \relcoeff{1}{3} \end{relmatrix}$ modulo $L'$, thus permutations of this relation simply say that swapping columns of $\begin{relmatrix}{2}{4} \relgen{1}{1}{1}{2} \relcoeff{2}{3} \relcoeff{2}{4} \end{relmatrix} - \begin{relmatrix}{2}{4} \relgen{2}{3}{2}{4} \relcoeff{1}{1} \relcoeff{1}{2} \end{relmatrix}$ is allowed modulo $L'$. There is no way of reducing $\rho_Q'$ to $0$ by swapping columns, so since these are all the relations we have at our disposal, $\rho_Q'$ is not in $L'$.
\end{proof}

\begin{remark}
For the determinant, relations of this multidegree are linear combinations of
\begin{multline*}
\left( \begin{relmatrix}{2}{4} \relgen{1}{2}{2}{3} \relcoeff{1}{1} \relcoeff{2}{4} \end{relmatrix} - \begin{relmatrix}{2}{4} \relgen{2}{3}{2}{4} \relcoeff{1}{1} \relcoeff{1}{2} \end{relmatrix} - \begin{relmatrix}{2}{4} \relgen{2}{2}{2}{4} \relcoeff{1}{1} \relcoeff{1}{3} \end{relmatrix} \right) - \left( \begin{relmatrix}{2}{4} \relgen{1}{2}{2}{3} \relcoeff{1}{1} \relcoeff{2}{4} \end{relmatrix} - \begin{relmatrix}{2}{4} \relgen{1}{1}{1}{2} \relcoeff{2}{3} \relcoeff{2}{4} \end{relmatrix} - \begin{relmatrix}{2}{4} \relgen{1}{1}{1}{3} \relcoeff{2}{2} \relcoeff{2}{4} \end{relmatrix} \right) \qquad \qquad \\
= \begin{relmatrix}{2}{4} \relgen{1}{1}{1}{2} \relcoeff{2}{3} \relcoeff{2}{4} \end{relmatrix} - \begin{relmatrix}{2}{4} \relgen{2}{3}{2}{4} \relcoeff{1}{1} \relcoeff{1}{2} \end{relmatrix} + \begin{relmatrix}{2}{4} \relgen{1}{1}{1}{3} \relcoeff{2}{2} \relcoeff{2}{4} \end{relmatrix} - \begin{relmatrix}{2}{4} \relgen{2}{2}{2}{4} \relcoeff{1}{1} \relcoeff{1}{3} \end{relmatrix},
\end{multline*}
which means that swapping columns of $\begin{relmatrix}{2}{4} \relgen{1}{1}{1}{2} \relcoeff{2}{3} \relcoeff{2}{4} \end{relmatrix} - \begin{relmatrix}{2}{4} \relgen{2}{3}{2}{4} \relcoeff{1}{1} \relcoeff{1}{2} \end{relmatrix}$ while \emph{also reversing sign} is allowed modulo $L$. It is possible to swap columns of this relation an odd number of times and get back what we started with --- for example, we swap columns $2$ and $3$, then $1$ and $2$, then $1$ and $3$, which overall swaps columns $1$ and $2$, and has no effect --- so for the determinant, we see that $\begin{relmatrix}{2}{4} \relgen{1}{1}{1}{2} \relcoeff{2}{3} \relcoeff{2}{4} \end{relmatrix} - \begin{relmatrix}{2}{4} \relgen{2}{3}{2}{4} \relcoeff{1}{1} \relcoeff{1}{2} \end{relmatrix}$ is equivalent to its negative modulo $L$, thus it must be equivalent to $0$ except in characteristic $2$. This is how the decomposition in \cref{thm:linear-relations-generate-quadratic} was constructed, and why it was necessary to assume that $\chartext(\bk) \neq 2$.
\end{remark}

For the permanent, this exception is the only major modification we need to make to \cref{thm:linear-relations-generate-quadratic}.  In other words, the argument of \cref{thm:linear-relations-generate-quadratic} establishes that in any characteristic, all quadratic relations are generated by linear relations $\rho'_1, \cdots, \rho'_5, \rho'_S$ together with the quadratic relation $\rho'_Q$ under permuting rows and columns and transposing. 

\subsection{Relations of higher degree}

The proof of \cref{thm:linear-generate-all-plural} needs no changes as long as the statement is amended to include the quadratic relations obtained from $\rho'_Q$ from permuting and transposing.  This establishes the following theorem:

\begin{theorem} \label{thm:linear+Q-generate-all-permanent}
All relations in the apolar ideal of the permanent (or the determinant if $\chartext(\bk) = 2$) are generated by the orbit of the linear relations $\rho'_1, \cdots, \rho'_5, \rho'_S$ from \eqref{E:linear-perm} and the orbit of the quadratic relation $\rho_Q'$ from \eqref{E:quadratic-perm} under permuting rows and columns and transposing. \epf
\end{theorem}

\begin{proposition} \label{thm:dimensions-perm}
The dimension of the vector space of linear relations is
\begin{equation*}
\beta_{2,3} = 4 \binom{n+1}{3} \binom{n+2}{3} = \frac{1}{9} n^2 \left( n+1 \right)^2 \left( n-1 \right) \left( n+2 \right).
\end{equation*}
The dimension of the vector space of quadratic relations modulo the subspace generated by the linear relations is
\begin{equation*}
\beta_{2,4} = 2 \binom{n}{2} \binom{n}{4} = \frac{1}{24} n^2 (n-1)^2 (n-2) (n-3).
\end{equation*}
All other Betti numbers $\beta_{2,j}$ for $j \neq 3,4$ are zero.
\end{proposition}

\begin{proof}
As already pointed out, \cref{thm:beta23} holds for $\perm_n$ with only minor modifications; this gives the dimension of the vector space of linear relations.  The additional quadratic relations are generated by the orbit of $\rho_Q'$ under $\bS_n \times \bS_n$ and transposing.
We saw in \cref{thm:quadratic-generator-for-perm} that permutations of the columns of this relation are equivalent modulo $L'$. Also, swapping the rows is the same as multiplying by $-1$ and permuting the columns. Therefore we need only compute the number of multidegrees with the partition form $(2+2, \ 1+1+1+1)$. This number is $\binom{n}{2} \binom{n}{4}$: we must choose $2$ of the $n$ rows to give multidegree $2$, and $4$ of the $n$ columns to give multidegree $1$.
After including a factor of $2$ to account for transposing, we have the dimension
\begin{equation*}
2 \binom{n}{2} \binom{n}{4}. 
\end{equation*}
The statement about the Betti numbers for $j \neq 3,4$ is \cref{thm:linear+Q-generate-all-permanent}.
\end{proof}

\begin{proof}[Proof of \cref{main-theorem}\ref{main-theorem-b}]
The theorem follows from combining \cref{thm:linear+Q-generate-all-permanent} and \cref{thm:dimensions-perm}.
\end{proof}

\section{Higher syzygies} \label{sec:higher-syzygies}

In this section, we compute $\beta_{3,4}$, the dimension of the space of linear second syzygies, for the apolar ideals $\det_n^{\perp}$ and $\perm_n^{\perp}$ (\cref{thm:beta34}) and provide a conjectural description of $\beta_{r,r+1}$, the dimension of the space of linear higher syzygies (\cref{conj:betarr1}) of $\det_n^{\perp}$.  We also record Macaulay2 computations of the graded Betti tables for small $n$ in \cref{sec:betti}.

\subsection{Linear second syzygies}
Using \cref{main-theorem}, we can compute the Betti number $\beta_{3,4}$, the number of linear second syzygies, of $S / \det_n^\perp$ and $S / \perm_n^\perp$.  

\begin{proposition} \label{thm:beta34}
The dimension of the space of linear second syzygies of $S / \det_n^\perp$ is
\begin{equation*}
\beta_{3,4} = \frac{1}{192} (n-1) n^2 (n+1)^2 (n+2) (5n^2 + 5n - 18) = 6 \binom{n+1}{4} \binom{n+3}{4} + 9 \binom{n+2}{4}^2.
\end{equation*}
The dimension of the space of linear second syzygies of $S / \perm_n^\perp$ is
\begin{multline*}
\beta_{3,4} = \frac{1}{192} (n - 1) n^2 (5 n^5 + 25 n^4 + 35 n^3 - 85 n^2 + 8n - 84) \\
= 6 \binom{n+1}{4} \binom{n+3}{4} + 9 \binom{n+2}{4}^2 + 2 \binom{n}{2} \binom{n}{4}.
\end{multline*}
\end{proposition}

\begin{proof}
Recall that if $M$ is a finitely generated graded $S$-module and $\cdots \to F_1 \to F_0 \to M \to 0$ is a minimal graded free resolution with $F_i = \bigoplus_j S(-j)^{\beta_{i,j}}$, then we have the following formulas for the Hilbert function $H_M$ of $M$ (c.f. \cite{eisenbud-syzygies})
\begin{equation} \label{E:hilbert-formula}
H_{M}(d) = \sum_i (-1)^i H_{F_i}(d) = \sum_{i,j} (-1)^i \beta_{i,j} H_S(d-j)
\end{equation}
in terms of the Hilbert functions $H_{F_i}$ and $H_S$.  Since $S$ is a polynomial ring in $n^2$ variables and using \cref{thm:dim-S/det}, we know that 
$$
H_S(d) = \binom{n^2 + d - 1}{d}
\qquad \text{and} \qquad
H_{S/\det_n^\perp}(d) = H_{S/\perm_n^\perp}(d) = \binom{n}{d}^2.
$$
For the determinant, \eqref{E:hilbert-formula} implies that
\begin{align*}
H_{S / \det_n^\perp}(4) & = \beta_{0,0} H_S(4) - \beta_{1,2} H_S(2) + \beta_{2,3} H_S(1) - \beta_{3,4} H_S(0)
\end{align*}
since all other $\beta_{i,j}$ are either $0$, or have $j > 4$ so $H_S(4-j) = 0$.

We know the value of each $\beta_{i,j}$ here (in particular, $\beta_{0,0} = 1$, and $\beta_{1,2}$ and $\beta_{2,3}$ are computed in \cref{thm:beta12,thm:beta23} respectively).
Solving this equation for $\beta_{3,4}$ gives the formula above.

For the permanent, $\beta_{2,4}$ is non-zero, so \eqref{E:hilbert-formula} implies that
\begin{align*}
H_{S / \perm_n^\perp}(4) & = \beta_{0,0} H_S(4) - \beta_{1,2} H_S(2) + \beta_{2,3} H_S(1) + \beta_{2,4} H_S(0) - \beta_{3,4} H_S(0)
\end{align*}
and solving this gives the formula above for $\beta_{3,4}$.
\end{proof}

The information we know about the Betti numbers of $S / \det_n^\perp$ and $S / \perm_n^\perp$ is summarized in \cref{tab:Betti-tables-det-perm}.

\begin{table}[htbp]
\caption{The known entries of the Betti tables of $S / \det_n^\perp$ and $S / \perm_n^\perp$} \label{tab:Betti-tables-det-perm}
\begin{subtable}{\linewidth}
\caption{Betti table for $S / \det_n^\perp$}
\begin{equation*} \renewcommand{\arraystretch}{1.3}
\begin{array}{r | c c c c c}
& 0 & 1 & 2 & 3 & \hdots \\ \hline
0 & 1 & 0 & 0 & 0 & \hdots \\
1 & 0 & \binom{n+1}{2}^2 & 4 \binom{n+1}{3} \binom{n+2}{3} & 6 \binom{n+1}{4} \binom{n+3}{4} + 9 \binom{n+2}{4}^2 \\
2 & 0 & 0 & 0 \\
\vdots & \vdots & \vdots & \vdots
\end{array}
\end{equation*}
\end{subtable}
\begin{subtable}{\linewidth}
\caption{Betti table for $S / \perm_n^\perp$}
\begin{equation*} \renewcommand{\arraystretch}{1.3}
\begin{array}{r | c c c c c}
& 0 & 1 & 2 & 3 & \hdots \\ \hline
0 & 1 & 0 & 0 & 0 & \hdots \\
1 & 0 & \binom{n+1}{2}^2 & 4 \binom{n+1}{3} \binom{n+2}{3} & 6 \binom{n+1}{4} \binom{n+3}{4} + 9 \binom{n+2}{4}^2 + 2 \binom{n}{2} \binom{n}{4} \\
2 & 0 & 0 & 2 \binom{n}{2} \binom{n}{4} \\
3 & 0 & 0 & 0 \\
\vdots & \vdots & \vdots & \vdots
\end{array}
\end{equation*}
\end{subtable}
\end{table}

The formulae for $\beta_{r,r+1}$ for the determinant seem to follow a pattern. We conjecture that the pattern continues:

\begin{conjecture} \label{conj:betarr1}
The Betti number $\beta_{r,r+1}$ for $S / \det_n^\perp$ is given by
\begin{equation*}
\beta_{r,r+1} = r \sum_{i=1}^r N_{r,i} \binom{n+i}{r+1} \binom{n+r-i+1}{r+1}
\end{equation*}
where $N_{r,i}$ are the Narayana numbers (sequence \href{http://oeis.org/A001263}{A001263} at \cite{OEIS}):
\begin{align*}
N_{r,i} & = \frac{1}{r} \binom{r}{i} \binom{r}{i-1}.
\end{align*}
\end{conjecture}

We know this conjecture is true for $r \leq 3$ by \cref{thm:beta12,thm:beta23,thm:beta34}.  The first few Narayana numbers are shown in \cref{tab:Narayana-numbers}.   See \cref{conj:repn} for a refinement of this conjecture.

\begin{table}[htbp]
\caption{Table of Narayana numbers $N_{r,i}$ for small $r,i$. Numbers in italics are known to agree with the formulas and Betti tables; the rest are conjectural.}\label{tab:Narayana-numbers}
\begin{equation*}
\begin{array}{r r | c c c c c c c}
&& \multicolumn{7}{c}{i} \\
&& 1 & 2 & 3 & 4 & 5 & 6 & 7 \\ \hline
\multirow{7}{*}{r} & 1 & \mathit{1} \\
& 2 & \mathit{1} & \mathit{1} \\
& 3 & \mathit{1} & \mathit{3} & \mathit{1} \\
& 4 & \mathit{1} & \mathit{6} & \mathit{6} & \mathit{1} \\
& 5 & 1 & \mathit{10} & \mathit{20} & \mathit{10} & 1 \\
& 6 & 1 & 15 & \mathit{50} & \mathit{50} & 15 & 1 \\
& 7 & 1 & 21 & 105 & \mathit{175} & 105 & 21 & 1
\end{array}
\end{equation*}
\end{table}

\subsection{Betti tables $S / \det_n^\perp$ and $S / \perm_n^\perp$ for low $n$} \label{sec:betti}
Using the Macaulay2 software on our laptops\footnote{Thanks to the help of Scott Morrison, we also tried these computations on a more powerful computer, but unfortunately this didn't allow for the computation of any additional Betti numbers.} we computed the graded Betti tables of $S / \det_n^\perp$ and $S / \perm_n^\perp$.   Some of these tables are incomplete due to computer limitations; where this is the case, every column shown is complete, but there may be more columns that aren't shown.  The computations presented in \crefrange{tab:Betti-2}{tab:Betti-7} support \cref{conj:betarr1}.

\begin{table}[htbp]
\caption{Complete Betti tables for $S / \det_n^\perp$ and $S / \perm_n^\perp$ where $n = 2$.}\label{tab:Betti-2}
\begin{subtable}[b]{.5\linewidth}
\centering
\caption{$S / \det_2^\perp$}\label{tab:Betti-det2}

\begin{tabular}{ r | *{5}{c} }
 & 0 & 1 & 2 & 3 & 4 \\ \hline
0 & 1 \\
1 & & 9 & 16 & 9 \\
2 & & & & & 1
\end{tabular}

\end{subtable}%
\begin{subtable}[b]{.5\linewidth}
\centering
\caption{$S / \perm_2^\perp$}\label{tab:Betti-perm2}

\begin{tabular}{ r | *{5}{c} }
 & 0 & 1 & 2 & 3 & 4 \\ \hline
0 & 1 \\
1 & & 9 & 16 & 9 \\
2 & & & & & 1
\end{tabular}

\end{subtable}%
\end{table}

\begin{table}[htbp]
\caption{Complete Betti tables for $S / \det_n^\perp$ and $S / \perm_n^\perp$ where $n = 3$} \label{tab:Betti-3}
\begin{subtable}[t]{\linewidth}
\centering
\caption{$S / \det_3^\perp$}\label{tab:Betti-det3}
\begin{tabular}{ r | *{10}{c} }
& 0 & 1 & 2 & 3 & 4 & 5 & 6 & 7 & 8 & 9 \\ \hline
0 & 1 \\
1 & & 36 & 160 & 315 & 288 & 100 \\
2 & & & & & 100 & 288 & 315 & 160 & 36 \\
3 & & & & & & & & & & 1
\end{tabular}
\end{subtable} \newline

\begin{subtable}[t]{\linewidth}
\centering
\caption{$S / \perm_3^\perp$}\label{tab:Betti-perm3}
\begin{tabular}{ r | *{10}{c} }
& 0 & 1 & 2 & 3 & 4 & 5 & 6 & 7 & 8 & 9 \\ \hline
0 & 1 \\
1 & & 36 & 160 & 315 & 288 & 116 \\
2 & & & & & 116 & 288 & 315 & 160 & 36 \\
3 & & & & & & & & & & 1
\end{tabular}
\end{subtable}
\end{table}

\setlength{\rotFPtop}{\fill}
\setlength{\rotFPbot}{\fill}

\begin{sidewaystable}
\caption{Complete Betti tables for $S / \det_n^\perp$ and $S / \perm_n^\perp$ where $n = 4$} \label{tab:Betti-4}
\begin{subtable}[t]{\linewidth}
\centering
\caption{$S / \det_4^\perp$}\label{tab:Betti-det4}
{\setlength{\tabcolsep}{4pt} \begin{tabular}{ r | *{18}{c} l }
& 0 & 1 & 2 & 3 & 4 & 5 & 6 & 7 & 8 & 9 & 10 & 11 & 12 & 13 & 14 & 15 & 16 \\ \hline
0 & 1 \\
1 & & 100 & 800 & 3075 & 6496 & 7700 & 4800 & 1225 \\
2 & & & & & 2500 & 16800 & 51275 & 93600 & 113256 & 93600 & 51275 & 16800 & 2500 \\
3 &&&&&&&&&& 1225 & 4800 & 7700 & 6496 & 3075 & 800 & 100 \\
4 &&&&&&&&&&&&&&&&& 1
\end{tabular}}
\end{subtable}
\newline

\begin{subtable}[t]{\linewidth}
\centering
\caption{$S / \perm_4^\perp$}\label{tab:Betti-perm4}
{\setlength{\tabcolsep}{4pt} \begin{tabular}{ r | *{18}{c} l }
& 0 & 1 & 2 & 3 & 4 & 5 & 6 & 7 & 8 & 9 & 10 & 11 & 12 & 13 & 14 & 15 & 16 \\ \hline
0 & 1 \\
1 & & 100 & 800 & 3087 & 6688 & 8400 & 4320 & 794 \\
2 &&& 12 & 192 & 3200 & 16320 & 50844 & 93600 & 113256 & 93600 & 50844 & 16320 & 3200 & 192 & 12 \\
3 &&&&&&&&&& 794 & 4320 & 8400 & 6688 & 3087 & 800 & 100 \\
4 &&&&&&&&&&&&&&&&& 1
\end{tabular}}
\end{subtable}
\end{sidewaystable}

\clearpage

\begin{table}[htbp]
\caption{Partial Betti tables for $S / \det_n^\perp$ and $S / \perm_n^\perp$ where $n = 5$} \label{tab:Betti-5}
\begin{subtable}[t]{0.5\linewidth}
\centering
\caption{$S / \det_5)^\perp$}\label{tab:Betti-det5}
\begin{tabular}{ r | *{4}{c} l }
& 0 & 1 & 2 & 3 & $\dotsm$ \\ \hline
0 & 1 \\
1 & & 225 & 2800 & 17325 \\
$\vdots$
\end{tabular}
\end{subtable}%
\begin{subtable}[t]{0.5\linewidth}
\centering
\caption{$S / \perm_5)^\perp$}\label{tab:Betti-perm5}
\begin{tabular}{ r | *{4}{c} l }
& 0 & 1 & 2 & 3 & $\dotsm$ \\ \hline
0 & 1 \\
1 & & 225 & 2800 & 17425 \\
2 &&& 100 & 2400 \\
$\vdots$
\end{tabular}
\end{subtable}
\end{table}

\begin{table}[htbp]
\caption{Partial Betti tables for $S / \det_n^\perp$ and $S / \perm_n^\perp$ where $n = 6$} \label{tab:Betti-6}
\begin{subtable}[t]{0.5\linewidth}
\centering
\caption{$S / \det_6^\perp$}\label{tab:Betti-det6}
\begin{tabular}{ r | *{3}{c} l }
& 0 & 1 & 2 & $\dotsm$ \\ \hline
0 & 1 \\
1 & & 441 & 7840 \\
$\vdots$
\end{tabular}
\end{subtable}%
\begin{subtable}[t]{0.5\linewidth}
\centering
\caption{$S / \perm_6^\perp$}\label{tab:Betti-perm6}
\begin{tabular}{ r | *{3}{c} l }
& 0 & 1 & 2 & $\dotsm$ \\ \hline
0 & 1 \\
1 & & 441 & 7840 \\
2 &&& 450 \\
$\vdots$
\end{tabular}
\end{subtable}
\end{table}

\begin{table}[htbp]
\caption{Partial Betti tables for $S / \det_n^\perp$ and $S / \perm_n^\perp$ where $n = 7$} \label{tab:Betti-7}
\begin{subtable}[t]{0.5\linewidth}
\centering
\caption{$S / \det_7^\perp$}\label{tab:Betti-det7}
\begin{tabular}{ r | *{3}{c} l }
& 0 & 1 & 2 & $\dotsm$ \\ \hline
0 & 1 \\
1 & & 784 & 18816 \\
$\vdots$
\end{tabular}
\end{subtable}%
\begin{subtable}[t]{0.5\linewidth}
\centering
\caption{$S / \perm_7^\perp$}\label{tab:Betti-perm7}
\begin{tabular}{ r | *{3}{c} l }
& 0 & 1 & 2 & $\dotsm$ \\ \hline
0 & 1 \\
1 & & 784 & 18816 \\
2 &&& 1470 \\
$\vdots$
\end{tabular}
\end{subtable}
\end{table}

\section{Representation-theoretic description of syzygy modules} \label{sec:repn-theory}
In this section, we will assume that $\chartext(\bk) = 0$.  Denote $V \cong W \cong \bk^n$.  Recall from \cref{sec:symmetries} that the symmetry group of $\det_n$ is  $G_{\det_n} = (\GL(V) \times \GL(W))/\bk^* \rtimes \ZZ/2$ where the non-identity element of $\ZZ/2$ corresponds to the \emph{transpose element} of $G_{\det_n}$ (i.e., $M_n(\bk) \to M_n(\bk), A \mapsto A^\top$). The minimal free resolution \eqref{E:free-det} of $S/\det_n^{\perp}$ is naturally a resolution of $G_{\det}$-representations.  The main results of this section are \cref{thm:generators-repn,thm:relations-repn,thm:2ndsyzygies-repn}, which provide a representation-theoretic characterization of the space of generators, the space of relations, and the space of linear second syzygies.  Moreover, we provide a conjectural representation-theoretic description of the space of linear higher syzygies of the determinant.

\subsection{Representation theory}
 A representation of $G_{\det_n}$ is the data of a $\GL(V) \times \GL(W)$-representation $Q$ such that the diagonal $\bk^* \subset \GL(V) \times \GL(W)$ (consisting of pairs $(\alpha I_n, \alpha^{-1} I_n)$ for $\alpha \in \bk^*$) acts trivially together with an involution $\iota \co Q \to Q$ (corresponding to transpose element of $G_{\det_n}$) 
such that $(g_1,g_2) \cdot \iota(q) = \iota( (g_2, g_1) \cdot q)$ for $(g_1,g_2) \in \GL(V) \times \GL(W)$ and $q \in Q$.
By standard representation theory, any irreducible representation of $\GL(V) \times \GL(W)$ is $V_{\lambda} \tensor W_{\eta}$ where $\lambda: \lambda_1 \ge \cdots \ge \lambda_n$ and $\eta \co \eta_1 \ge \cdots \ge \eta_n$ are the \emph{highest weights}.  We denote by $|\lambda| = \lambda_1 + \cdots + \lambda_n$ the \emph{degree} of a weight $\lambda$.
Thus an irreducible representation $G_{\det_n}$ either has the form $V_{\lambda} \tensor W_{\lambda}$ for a highest weight $\lambda$ or $(V_{\lambda} \tensor W_{\eta}) \oplus (V_{\eta} \tensor W_{\lambda})$ for a pair of distinct highest weights $\lambda, \eta$ with $|\lambda| = |\eta|$.

\subsection{The space of generators}
By \cref{thm:Shafiei}, we know that $\det_n^{\perp}$ is generated by $X_{i,i}^2$, $X_{i,j}X_{i,k}$, $X_{i,j}X_{k,j}$, and $X_{i,j} X_{k,l} + X_{i,l} X_{k,j}$.  This decomposes the degree $2$ component $F_1$ as a direct sum of one-dimension representations of $T_V \times T_W$, where $T_V$ denote the maximal torus of $\GL(V)$.  By examining the weights, we see that $X_{1,1}^2$ is the highest weight vector with weight $\big((2), (2)\big)$.    Note that $V_{(2)} \tensor W_{(2)} \cong \Sym^2 V \tensor \Sym^2 W$ is an irreducible $G_{\det_n}$-representation where the transpose automorphism acts on this representation by interchanging $\Sym^2 V$ and $\Sym^2 W$.
But since $\dim(\Sym^2 V \tensor \Sym^2 W) = {n+1 \choose 2}^2$ is equal to $\beta_{1,2}$ (i.e., the number of minimal quadratic generators of $\det_n^{\perp}$), we conclude that:
\begin{proposition} \label{thm:generators-repn}
There is an isomorphism of $G_{\det_n}$-representations
\begin{equation*}
F_1 \cong (\Sym^2 V \tensor \Sym^2 W) \tensor S
\end{equation*}
In particular, the degree $2$ component of $F_1$ is isomorphic to the irreducible $G_{\det_n}$-representation $\Sym^2 V \tensor \Sym^2 W$.
\end{proposition}

\subsection{The space of relations}
The linear relations have standard degree $3$. The highest weight possible is $\big((3), (3)\big)$; however, no non-zero relations have this weight, since there is only one element of $F_1$ with this multidegree, namely $X_{1,1} (X_{1,1}^2)$. The next highest weight possible is $\big((3), (2,1)\big)$, and this weight space is inhabited by $\rho_1 = X_{1,2} (X_{1,1}^2) - X_{1,1} (X_{1,1} X_{1,2})$. Note that $\big((2,1), (3)\big)$ is inhabited by the transpose of this.  We see that as $\GL(V) \times \GL(W)$-representations, the degree $3$ component of $F_2$ contains $(V_{(3)} \tensor W_{(2,1)}) \oplus (V_{(2,1)} \tensor W_{(3)})$.  Since $\dim V_{(3)} = {n+2 \choose 3}$ and $\dim V_{(2,1)} = 2 {n+1 \choose 3}$, we compute that
\begin{equation*}
\dim (V_{(3)} \tensor W_{(2,1)}) \oplus (V_{(2,1)} \tensor W_{(3)}) = 4 {n+2 \choose 3} {n+1 \choose 3}
\end{equation*}
which is equal to $\beta_{2,3}$.  We conclude that:

\begin{proposition} \label{thm:relations-repn}
There is an isomorphism of $G_{\det_n}$-representations
\begin{equation*}
F_2 \cong \big((V_{(3)} \tensor W_{(2,1)}) \oplus (V_{(2,1)} \tensor W_{(3)})\big) \tensor S
\end{equation*}
In particular, the degree $3$ component of $F_2$ is isomorphic to the irreducible $G_{\det_n}$-representation $(V_{(3)} \tensor W_{(2,1)}) \oplus (V_{(2,1)} \tensor W_{(3)})$.
\end{proposition}

\subsection{The space of linear 2nd syzygies}
The linear second syzygies have standard degree $4$.  It's not hard to check that there are no second syzygies of weight
$\big((4), (4)\big)$, $\big((4), (3,1)\big)$ and $\big((4), (2,2)\big)$.  The second syzygy 
\begin{equation*}
X_{1,3} \left( \begin{relmatrix}{1}{3} \relgen{1}{1}{1}{1} \relcoeff{1}{2} \end{relmatrix} - \begin{relmatrix}{1}{3} \relgen{1}{1}{1}{2} \relcoeff{1}{1} \end{relmatrix} \right) - X_{1,2} \left( \begin{relmatrix}{1}{3} \relgen{1}{1}{1}{1} \relcoeff{1}{3} \end{relmatrix} - \begin{relmatrix}{1}{3} \relgen{1}{1}{1}{3} \relcoeff{1}{1} \end{relmatrix} \right) + X_{1,1} \left( \begin{relmatrix}{1}{3} \relgen{1}{1}{1}{2} \relcoeff{1}{3} \end{relmatrix} - \begin{relmatrix}{1}{3} \relgen{1}{1}{1}{3} \relcoeff{1}{2} \end{relmatrix} \right)
\end{equation*}
has weight $\big((4), (2,1,1)\big)$ and its transpose has weight $\big((2,1,1), (4)\big)$.  This generates an irreducible $G_{\det_n}$-representation $\big( (V_{(2,1,1)} \tensor W_{(4)}) \oplus (V_{(4)} \tensor W_{(2,1,1)}) \big)$ contained in the  space of linear second syzygies (i.e., the degree $4$ component of $F_3$).  
The next highest weight appearing in the space of linear second syzygies but not in this subrepresentation is $\big((3,1), (3,1)\big)$, which is inhabited by
\begin{equation*}
X_{2,1} \left( \begin{relmatrix}{2}{2} \relgen{1}{1}{1}{1} \relcoeff{1}{2} \end{relmatrix} - \begin{relmatrix}{2}{2} \relgen{1}{1}{1}{2} \relcoeff{1}{1} \end{relmatrix} \right) - X_{1,2} \left( \begin{relmatrix}{2}{2} \relgen{1}{1}{1}{1} \relcoeff{2}{1} \end{relmatrix} - \begin{relmatrix}{2}{2} \relgen{1}{1}{2}{1} \relcoeff{1}{1} \end{relmatrix} \right) + X_{1,1} \left( \begin{relmatrix}{2}{2} \relgen{1}{1}{1}{2} \relcoeff{2}{1} \end{relmatrix} - \begin{relmatrix}{2}{2} \relgen{1}{1}{2}{1} \relcoeff{1}{2} \end{relmatrix} \right). \qedhere
\end{equation*}
We can therefore conclude:

\begin{proposition} \label{thm:2ndsyzygies-repn}
The degree $4$ component of $F_3$ is isomorphic as $G_{\det_n}$-representations to 
\begin{equation*}
\big( (V_{(2,1,1)} \tensor W_{(4)}) \oplus (V_{(4)} \tensor W_{(2,1,1)}) \big) \oplus  (V_{(3,1)} \tensor W_{(3,1)}),
\end{equation*}
the direct sum of two irreducible $G_{\det_n}$-representations.
\end{proposition}

\subsection{Conjectural description}

Observing the pattern of \cref{thm:generators-repn,thm:relations-repn,thm:2ndsyzygies-repn}, we can conjecture:

\begin{conjecture} \label{conj:repn}
The degree $r+1$ component of $F_r$ is isomorphic as a $G_{\det_n}$-representation to
\begin{equation} \label{E:conj-repn}
\bigoplus_{i=1}^r V_{(r-i+2, 1^{i-1})} \otimes V_{(i+1, 1^{r-i})},
\end{equation}
where $1^a$ means $a$ copies of $1$. This is the direct sum of $\lfloor \frac{r+1}{2} \rfloor$ irreducible $G_{\det_n}$-representations.
\end{conjecture}

\begin{remark}
The dimension of \eqref{E:conj-repn} agrees with the dimension in \cref{conj:betarr1}. Indeed, the dimension of the irreducible representation corresponding to a Young diagram $\lambda$ can be given by the hook length formula:
\begin{equation*}
\dim V_\lambda = \prod_{(i,j) \in \lambda} \frac{i - j + n}{\hook(\lambda_{i,j})}
\end{equation*}
where the product is over cells $(i,j)$ in the Young diagram. For a Young diagram shaped like a $\Gamma$ with $c$ cells, of which $d$ are in the top row, all cells are either the top left cell, or a cell in the rightward arm, or a cell in the downward column, so this formula simplifies to: 
\begin{align*}
\dim V_{(d, 1^{c-d})} & = \left( \frac{n}{c} \right) \cdot \left( \frac{(n+1)}{(d-1)} \dotsm \frac{(n+d-1)}{1} \right) \cdot \left( \frac{(n-1)}{(c-d)} \dotsm \frac{(n-c+d)}{1} \right) \\
& = \frac{(n+d-1)! (n-1)!}{c (n-1)! (n-c+d-1)! (d-1)! (c-d)!} \\
& = \binom{c-1}{d-1} \binom{n+d-1}{c}
\end{align*}

Summing this expression over the direct sum in \cref{conj:repn} gives
\begin{align*}
\dim (F_r)_{r+1} & = \sum_{i=1}^r \binom{r}{r-i+1} \binom{n+r-i+1}{r+1} \cdot \binom{r}{i} \binom{n+i}{r+1} \\
& = r \sum_{i=1}^r \frac{1}{r} \binom{r}{i} \binom{r}{i-1} \binom{n+i}{r+1} \binom{n+r-i+1}{r+1}
\end{align*}
which is the expression in \cref{conj:betarr1}.
\end{remark}

\end{document}